\documentclass{amsart}
\usepackage[dvipsnames]{xcolor}
\usepackage{hyperref}
\hypersetup{
     colorlinks=true,
     linkcolor=blue!60!black,
     filecolor=blue!75!black,
     citecolor = green!50!black,      
     urlcolor=magenta,
     }
     
\usepackage{ stmaryrd }
\usepackage{amssymb}
\usepackage{ bbold }
\usepackage[all]{xy}
\xyoption{all}
\usepackage{url}
\usepackage{mathrsfs}
\usepackage{tikz-cd}
\usepackage{eucal}
\usepackage{comment}
\usepackage{mathtools}
\usepackage[nameinlink,capitalise,noabbrev]{cleveref}

\counterwithin{equation}{section}
\usepackage{todonotes}

\title[Geometric purity]{Geometric purity and the frame of smashing ideals}

\newcommand{\comments}[1]{}

\newcommand{\colim}{\operatorname{colim}\nolimits}

\newcommand{\h}{\operatorname{h}\nolimits}

\newcommand{\Hom}{\operatorname{Hom}\nolimits}

\newcommand{\Mod}{\operatorname{Mod}\nolimits}

\newcommand{\Spc}{\operatorname{Spc}\nolimits}

\newcommand{\supp}{\operatorname{supp}\nolimits}

\newcommand{\loc}{\operatorname{loc}\nolimits}

\DeclareMathOperator{\Zg}{Zg}

\DeclareMathOperator{\mmod}{mod}
\DeclareMathOperator{\gpinj}{gpinj}
\DeclareMathOperator{\pinj}{pinj}

\newcommand{\mmmod}{\mmod\kern-0.1em\text{-}}%
\newcommand{\MMod}{\Mod\kern-0.1em\text{-}}%
\newcommand{\ModT}{\MMod \T^\omega}

\def \A{{\mathcal A}}

\def \pp{\mathbf p }

\def \D{{\mathcal D}}
\def \e{{\mathbb e}}

\def \f{{\mathbb f}}

\def \I{{\mathcal I}}

\def \K{{\mathcal K}}
\def \L{{\mathcal L}}

\def \P{{\mathcal P}}
\def \pp{\mathfrak p }
\def \Q{{\mathcal Q}}

\def \S{{\mathcal S}}
\def \T{{\mathcal T}}
\def \U{{\mathcal U}}
\def \Z{{\mathbb Z}}

\makeatletter
\newcommand*{\doublerightarrow}[2]{\mathrel{
  \settowidth{\@tempdima}{$\scriptstyle#1$}
  \settowidth{\@tempdimb}{$\scriptstyle#2$}
  \ifdim\@tempdimb>\@tempdima \@tempdima=\@tempdimb\fi
  \mathop{\vcenter{
    \offinterlineskip\ialign{\hbox to\dimexpr\@tempdima+1em{##}\cr
    \rightarrowfill\cr\noalign{\kern.5ex}
    \rightarrowfill\cr}}}\limits^{\!#1}_{\!#2}}}
\newcommand*{\triplerightarrow}[1]{\mathrel{
  \settowidth{\@tempdima}{$\scriptstyle#1$}
  \mathop{\vcenter{
    \offinterlineskip\ialign{\hbox to\dimexpr\@tempdima+1em{##}\cr
    \rightarrowfill\cr\noalign{\kern.5ex}
    \rightarrowfill\cr\noalign{\kern.5ex}
    \rightarrowfill\cr}}}\limits^{\!#1}}}
\makeatother

\newcommand{\nc}{\newcommand}
\nc{\dmo}{\DeclareMathOperator}
\nc{\Weyl}[2]{{#1}/\!\!/{#2}}
\nc{\WGH}{\Weyl{G}{H}}
\nc{\WGK}{\Weyl{G}{K}}
\nc{\WGL}{\Weyl{G}{L}}
\nc{\WGN}{\Weyl{G}{N}}

\numberwithin{equation}{section}
\newtheorem{theorem}[equation]{Theorem}
\newtheorem{proposition}[equation]{Proposition}
\newtheorem{corollary}[equation]{Corollary}
\newtheorem{lemma}[equation]{Lemma}

\theoremstyle{remark}
\newtheorem{definition}[equation]{Definition}
\newtheorem{terminology}[equation]{Terminology}
\newtheorem{remark}[equation]{Remark}
\newtheorem{recollection}[equation]{Recollection}
\newtheorem{example}[equation]{Example}

\newtheorem{notation}[equation]{Notation}

\newtheorem{convention}[equation]{Convention}

\newtheorem{warning}[equation]{Warning}
\newtheorem{question}[equation]{Question}

\thanks{}
\date{\today}

\begin{document}

\author[G\'omez]{Juan Omar G\'omez}
\author[Medina]{Mauricio Medina-Barcenas}
\author[Stevenson]{Greg Stevenson}
\author[Villarreal]{Bernardo Villarreal}
\author[Zaldivar]{Angel Zaldivar}

\address{Juan Omar G\'omez, Fakultat f\"ur Mathematik, Universit\"at Bielefeld, D-33501 Bielefeld, Germany}
\email{jgomez@math.uni-bielefeld.de}
\urladdr{https://sites.google.com/cimat.mx/juanomargomez/home}

\address{Mauricio Medina-Barcenas, Departamento de Matemáticas, Universidad Autónoma     Me-tropolitana UAM
Iztapalapa, CDMX, M\'exico}
\email{mmedina@xanum.uam.mx}
\urladdr{https://www.medinabarcenas.com/}

\address{Greg Stevenson, Aarhus University, Department of Mathematics, Ny Munkegade 118, bldg. 1530, 8000 Aarhus C, Denmark}
\email{greg@math.au.dk}
\urladdr{https://sites.google.com/view/gregstevenson}

\address{Bernardo Villarreal, Departamento de Matematicas, Centro de Investigaci\'on y de Estudios Avanzados del Instituto Politecnico Nacional, Av. IPN 2508, CDMX, M\'exico
}
\email{bvillarreal@math.cinvestav.mx}
\urladdr{https://sites.google.com/math.cinvestav.edu.mx/bvillarreal}

\address{Angel Zaldivar, Department of Mathematics, University of Guadalajara, Blvd. Gral. Marcelino Grac\'ia 1421, Ol\'impica, 44430, Jalisco, M\'exico.
}
\email{luis.zaldivar@academicos.udg.mx}
\urladdr{https://www.lazaldivar.com/}

\begin{abstract}
We introduce the notion of geometric purity in rigidly-compactly generated tt-categories by considering exact triangles that are pure at each tt-stalk. We develop a systematic study of this concept, including examples and applications. In particular, we show that geometric purity is, in general, strictly stronger than ordinary purity, and that it naturally leads to the notion of geometrically pure-injective objects. We prove that such objects arise as pushforwards of pure-injective objects from suitable tt-stalks. Moreover, we give a detailed analysis of indecomposable geometrically pure-injective objects in the derived category of the projective line. Under mild additional assumptions, we identify the geometric part of the Ziegler spectrum as a closed subset.

As an application, we demonstrate that this new notion of purity can be used to tackle the problem of spatiality of the frame of smashing ideals via the geometric Ziegler spectrum. In particular, we show that our approach rules out the counterexamples of Balchin and Stevenson to existing methods. 
\end{abstract}

\maketitle
\tableofcontents
\setcounter{tocdepth}{1}

\setlength{\parindent}{0cm}
\setlength{\parskip}{0.8ex}

\section{Introduction}

Since its development in \cite{balmer2005spectrum}, the field of tensor triangular (tt) geometry has attracted a great deal of interest across several areas of mathematics. On the one hand, it provides a coherent framework for tackling classification problems in tt-categories that are typically considered \textit{wild}; on the other hand, it allows tools and techniques to be transported and unified across these different areas. We refer the reader to \cite{Bal10}, \cite{stevenson2018tour}, and \cite{balmer2020guide} for surveys of the subject. 

In practice, one often encounters rigidly-compactly generated tt-categories, and a natural problem is to understand the class of their localizing tensor ideals. Among these, smashing localizations play a distinguished role: they are central not only in tensor triangular geometry, but also in the framework of \textit{condensed mathematics} \cite{Sch19} and in the theory of compactly generated triangulated categories \cite{Kra00}.

We focus on rigidly-compactly generated tt-categories. In this setting, the class of smashing tensor ideals forms a \textit{frame}, as shown in \cite{balmer2020frame}. Recall that a frame is \textit{spatial} if it is isomorphic to the frame of open subsets of a topological space.  A natural question is the following.

\begin{question}\label{question}
Let $\T$ be a rigidly-compactly generated tt-category. Is the frame of smashing tensor ideals of $\T$ spatial?
\end{question}

Although a positive answer to this question is not strictly necessary in order to exploit this frame, it would significantly simplify several aspects of tt-geometry; see \cite{balchin2021big}. In fact, an early version of \textit{loc.\ cit.} claimed that this frame is always spatial, and a different argument appeared in \cite{Wag23}. However, both approaches were later shown to be incorrect, as clarified in \cite[Appendix A]{balchin2021big}. Consequently, the question remains open. It is known that without the rigidity assumption the frame of smashing ideals need not be spatial. This is shown in \cite{aoki2023sheaves}, which also develops the theory of the smashing frame in general from an $\infty$-categorical perspective.

It is worth emphasizing that the approach in \cite{Wag23} is motivated by model-theoretic techniques and, in particular, proposes a topological space whose frame of open subsets might be isomorphic to the frame of smashing ideals. Such techniques have recently found further applications in tt-geometry; see, for example, \cite{BW23} and \cite{PW24}, with \cite{Kra00} as an important precursor in the context of compactly generated triangulated categories.

This work is motivated by \cref{question}, and we take both \cite{balchin2021big} and \cite{Wag23} as our starting points. Informally, one may say that the difficulty in Wagstaffe’s argument is that the Ziegler spectrum has \textit{too many points}. One way to address this issue is to take the monoidal structure into account. To this end, we introduce the notion of geometric purity in a rigidly-compactly generated tt-category, which allows us to eliminate certain \textit{problematic} points from the Ziegler spectrum. This notion is inspired by ideas of Prest and Sl\'avik \cite{SP18}. 

We provide a detailed study of this new concept, supply examples, and ultimately use it to develop a method for addressing \cref{question}. We now describe our main contributions more precisely.

\subsection*{Main results}

Let $\T$ be a rigidly-compactly generated tt-category, and let $\P$ be a point of $\Spc(\T^\omega)$. The \textit{tt-stalk of $\T$ at $\P$} is defined as
\[
\T_\P\coloneqq \T/\mathrm{loc}(\P) 
\]
which is a rigidly-compactly generate tt-category, where $\mathrm{loc}(\P)$ denotes the smallest localizing subcategory of $\T$ generated by $\P$. We write $\iota_\P^\ast$ for the corresponding quotient functor $\T \to \T_\P$.

We are now in a position to introduce the central definition of this work.

\begin{definition}
An exact triangle $x \to y \to z$ in $\T$ is called \textit{geometrically pure} (or \textit{g-pure}, for short) if the induced exact triangle
\[
\iota_\P^\ast (x) \to \iota_\P^\ast(y) \to \iota_\P^\ast(z)
\] 
 is pure in $\T_\P$ for every $\P \in \Spc(\T^\omega)$. An object $x \in \T$ is said to be \textit{g-pure-injective} if every g-pure exact triangle $x\to y \to z$ splits. 
\end{definition}

One of the main contributions of this work is a systematic development of geometric purity in this setting. From the definition it follows immediately that every pure-triangle is g-pure. However, it is not a priori clear that g-purity is genuinely stronger than the usual notion of purity in compactly generated triangulated categories. We provide an example showing that this is indeed the case in \cref{g-pure vs g-pure á la PS}. We also discuss situations where the two notions coincide; for instance, in the derived category of a commutative ring (see  \cref{pure=gpure affine}).

Turning to indecomposable g-pure injective objects, we prove that all such objects can be realized as pushforwards of pure-injectives from tt-stalks. More precisely, this is recorded in the following result, which later appears as  \cref{thm-gp-stalks}.

\begin{theorem}
Let $\T$ be a rigidly-compactly generated tt-category, and let $x$ be an indecomposable g-pure injective object in $\T$. Then there exists a prime $\P \in \Spc(\T^\omega)$ and a pure-injective object $y \in \T_\P$ such that 
\[
x\simeq \iota_{\P,\ast} y
\] 
where $\iota_{\P,\ast}$ denotes the right adjoint of $\iota_\P^\ast$.
\end{theorem}

For a given cover $\{U_i\mid i \in I\}$ of $\Spc(\T^\omega)$ by quasi-compact open subsets, we analyze the relation between purity in $\T$ and purity in each $\T(U_i)$. In particular, by considering the subspace $\mathrm{GZg}(\T)$ of the Ziegler spectrum of $\T$ determined by the isomorphism classes of indecomposable g-pure injective objects, we obtain the following result, which is a reformulation of  \cref{coro-surjection-gzg-stalks}.

\begin{theorem}
Let $\T$ be a rigidly-compactly generated tt-category, and let $\{U_i\mid i \in I\}$ be a cover of $\Spc(\T^\omega)$ by quasi-compact open subsets. Then the induced map on Ziegler spectra 
\[
\coprod_{i\in I}\mathrm{GZg}(\T(U_i)) \to \mathrm{GZg}(\T)
\]
is a topological quotient. Moreover, each component map $\mathrm{GZg}(\T(U_i)) \to \mathrm{GZg}(\T)$ is a topological embedding. 
\end{theorem}

The previous result, together with some additional assumptions, shows that the geometric part of the Ziegler spectrum is in fact a closed subset (see \cref{coro-gzg-closed}). This applies, in particular, to the derived category of a qcqs scheme (see  \cref{coro-zgz-scheme}). In light of these examples, it is natural to ask the following question:

\begin{question}
Does there exist a cover $\{U_i \mid i \in I\}$ of $\Spc(\T^\omega)$ by quasi-compact open subsets such that for each $i$ $\mathrm{Zg}(\T(U_i))$ coincides with its geometric part $\mathrm{GZg}(\T(U_i))$?
\end{question}

We do not know the answer to this question, but a positive one would imply that the geometric part of the Ziegler spectrum is always a closed subset.

Let us now turn to an application of the theory developed in this work. Denote by $\mathbb{S}^\otimes(\T)$ the frame of smashing tensor-ideals of $\T$. As already mentioned, Wagstaffe’s approach to \cref{question} in \cite{Wag23}  consisted in studying the frame of definable tensor ideals which is anti-equivalent to the frame of smashing ideals, and using them to define a topology on the Ziegler spectrum of $\T$. This is a natural idea, since a definable category is completely determined by the pure-injectives it contains by work of Krause \cite{Kra00}. However, the aforementioned counterexamples show that the Ziegler spectrum contains too many points for this method to succeed.

This difficulty was our original motivation for seeking a stronger notion of purity, one that would allow us to disregard certain pure-injectives. To present our main result in this direction, we now introduce some terminology.

\begin{definition}
Let $\gpinj(\T)$ denote the set of isomorphism classes of indecomposable g-pure-injective objects of $\T$. A subset $U \subset \gpinj(\T)$ is called \textit{tt-closed} if it is of the form 
\[U=\D\cap \gpinj(\T)\] 
for some definable tt-ideal $\D$ of $\T$.  
\end{definition}

We emphasize that the tt-closed subsets do not define a topology on the entire set of isomorphism classes of indecomposable pure-injectives $\pinj(\T)$. Our hope, however, is that they do induce a topology on the geometric part of the Ziegler spectrum. In this case, one would obtain that the frame of smashing ideals is isomorphic to the frame of open subsets of the resulting space, thereby verifying spatiality. In particular, this definition allows us to establish a partial local-to-global principle for the spatiality of the frame of smashing ideals. 

\begin{theorem}
Let $\T$ be a rigidly-compactly generated tt-category. Suppose that for each $\P \in \Spc(\T^\omega)$, the tt-closed subsets determine the closed subsets of a topology on $\gpinj(\T_\P) = \pinj(\T_\P)$. Then the tt-closed subsets determine a topology on $\gpinj(\T)$. In particular, the frame of smashing ideals $\mathbb{S}^\otimes(\T)$ is spatial.  
\end{theorem}

It is not difficult to show that spatiality of $\mathbb{S}^\otimes(\T)$ implies spatiality of $\mathbb{S}^\otimes(\T_\P)$ for any prime $\P \in \Spc(\T^\omega)$ (see \cref{prop-global-to-local}). Moreover, the preceding result extends in a similar way to quasi-compact open covers. 

Finally, we show in \cref{sec-examples} that Balchin-Stevenson's examples in \cite{balchin2021big} are rule out by this approach.

\subsection*{Acknowledgements}
We thank Mike Prest for his interest in this project. 
JOG was supported by  the Deutsche Forschungsgemeinschaft (Project-ID 491392403 – TRR 358). GS was supported by Danmarks Frie Forskningsfond (grant ID: 10.46540/4283-00116B). This project was partially  supported by SECIHTI-Project CBF2023-2024-2630.

\section{Preliminaries}

In this section, we recollect basic terminology and results from tt-geometry as well as basics on purity in compactly generated triangulated categories. Our main references for this section are \cite{balmer2005spectrum}, \cite{BKS19}, \cite{balchin2021big} and \cite{Kra00}.  

\subsection{The Balmer spectrum}

Let $\K$ be an essentially small tt-category. A triangulated subcategory $\I$ of $\K$ is \textit{thick} if it is closed under retracts; it is a \textit{$\otimes$-ideal} if it stable under the tensor product; and it is \textit{prime} if it is a proper thick $\otimes $-ideal which satisfies that if $x\otimes y\in \K$, then $x\in \K$ or $y\in \K$. 

\begin{recollection}\label{rec-balmer-spc}
    The \textit{Balmer spectrum} of $\K$ is defined as  
    \[
    \Spc(\K)\coloneqq \{\P\mid \P \mbox{ is prime } \}
    \] 
    equipped with the topology generated by the base of closed subsets  
    \[
    \supp(x)\coloneqq \{\P\in \Spc(\K)\mid x\not\in \P\}
    \]
    as $x$ runs over the objects of $\K$. One refers to $\supp(x)$ as the \textit{support} of $x$. The space $\Spc(\K)$ is indeed a spectral space. 
\end{recollection}

\begin{recollection}
    Let $F\colon \K\to \L$ be a tt-functor. Then we have an induced spectral map 
    \[
    \Spc(\L)\to \Spc(\K), \quad \P\mapsto F^{-1}\P.
    \]
    In other words, the Balmer spectrum is a functor from the opposite category of tt-categories and tt-functors to the category of spectral spaces. 
\end{recollection}

\begin{recollection}\label{rec-Balmer-class}
    A remarkable result due to Balmer is that $\Spc(\K)$ together with $\supp(-)$ provides a parametrization of \textit{radical} tt-ideals of $\K$; that is, those tt-ideals $\I$ such that if $x^{\otimes n}\in \I$, then $x\in \I$. More precisely, we have an order preserving bijection  
    \begin{center}
        \begin{tikzcd}
            \mathrm{Rad}(\K)  \arrow[r, shift left=.5ex,"\supp"] & \mathrm{Thom}(\Spc(\K))  \arrow[l, shift left=.5ex,"\K_{-}"]
        \end{tikzcd}
    \end{center}
    where $\mathrm{Thom}(\Spc(\K))$ denotes the poset of Thomason subsets of the spectrum $\Spc(\K)$; that is, those sets that are unions of closed subsets that have quasi-compact complement. The map $\K_{Y}$ is the radical ideal consisting of  those elements $x\in \K$ whose support is contained in $Y$.  On the other hand, $\supp(\I)$ is simply $\bigcup_{x\in \I}\supp(x)$. 
\end{recollection}

\begin{remark}
    If the category $\K$ is also rigid, then all tt-ideals are automatically radical. Hence the above bijection gives us a complete classification of tt-ideals in this case. 
\end{remark}

\begin{remark}\label{rem-hochster-dual}
 Let us point out that the Thomason subsets of $\Spc(\K)$ are in fact the open subsets of the \textit{Hochster dual} $\Spc(\K)^\vee$ of $\Spc(\K)$. The space  $\Spc(\K)^\vee$ has the same points as $\Spc(\K)$, and it is also a spectral space. 
\end{remark}

\subsection{Big categories and geometric functors}

We are interested in tt-categories that have small coproducts. We restrict ourselves to those often referred as \textit{big tt-categories}. Let us make this precise. 

\begin{recollection}
    A tt-category with small coproducts $\T$ is \textit{rigidly-compactly generated} if the (essentially small) collection of compact objects agrees with the collection of dualizable objects, and the smallest localizing subcategory of $\T$ containing $\T^\omega$ is $\T$, where $\T^\omega$ is the full subcategory of compact objects. We refer to \cite[Section 1]{BF11} for further details.  
\end{recollection}
\begin{example}
    The following are examples of rigidly-compactly generated tt-categories.
    \begin{enumerate}
        \item Let $X$ be a quasi-compact quasi-separated scheme. Then the derived category $\mathbf{D}_{\mathbf{qc}}(X)$ of complexes of  $\mathcal{O}_X$--modules with quasi-coherent cohomology, equipped with the derived tensor product $\otimes^L_{\mathcal{O}_X}$ is a rigidly-compactly generated tt-category. 
        \item Let $G$ be a finite group and let $k$ be a field of positive characteristic dividing $|G|$. Then the stable module category equipped with the monoidal structure induced by $\otimes_k$ with the diagonal $G$-action is a rigidly-compactly generated tt-category. 
    \end{enumerate}
\end{example}

\begin{terminology}
    A coproduct preserving tt-functor $f^\ast\colon \T\to \S$ between rigidly-compactly generated is referred as a \textit{geometric functor. }
\end{terminology}

\begin{remark}
   Note that a geometric functor in particular preserves compact objects (since any monoidal functor preserves dualizable objects).  This in particular gives us for \textit{free} a couple of adjunctions that we spell out in the following recollection.  
\end{remark}

\begin{recollection}\label{rec-geometric-functors}
    Let $f^\ast\colon \T\to \S$ be a geometric functor. Then $f^\ast$ has a right adjoint $f_\ast$ which has itself a right adjoint $f^!$. In other words, there is a chain of adjunctions $f^\ast\dashv f_\ast \dashv f^!$. Moreover, the adjunction $(f^\ast,f_\ast)$ satisfies the projection formula. That is, for $x\in \T$ and $y\in \S$ we have 
    \[
    f_\ast(f^\ast x\otimes y)\simeq x\otimes f_\ast y
    \]
we refer to  \cite{BDS16} for further details. 
\end{recollection}

\subsection{Smashing localizations}\label{subsection-smashingloc}

In this subsection we briefly recall some results from \cite{BF11}. 
Let $\T$ be a rigidly-compactly generated tt-category. For a class $\mathcal{E}$ of objects of $\T$ we write $\mathcal{E}^\perp$ to denote its right orthogonal complement in $\T$. 

\begin{definition}
    A \textit{smashing ideal} of $\T$ is a localizing tensor-ideal $\S$ of $\T$  which satisfies that the Verdier quotient $q^\ast\colon \T\to \T/\S$ exists and has a right adjoint $q_\ast$ which commutes with small coproducts.
\end{definition}

\begin{recollection}\label{rec-smashing-ideals}
    Let $\S$ be a smashing ideal of $\T$. Then there is a localization functor $(L_\S,\lambda)\colon \T\to \T$ and a colocalization functor $(\Gamma_\S,\gamma)\colon \T\to \T$ such that $\S=\mathrm{Ker}(L_\S)=\mathrm{Im}(\Gamma_\S)$ and $\S^\perp=\mathrm{Im}(L_\S)=\mathrm{Ker}(\Gamma_\S)$. Moreover, the following properties hold. 
    \begin{enumerate}
        \item There is an idempotent functorial triangle 
    \[
     \Gamma_\S \xrightarrow[]{\gamma} \mathrm{Id}\xrightarrow[]{\lambda} L_\S. 
    \] 
    \item The localization functor $L_\S$ can be identified with the functor $L_\S(\mathbb 1)\otimes -$ and the colocalization functor $\Gamma_\S$ with $\Gamma_\S(\mathbb 1)\otimes -$. 
    \item $\S^\perp$ is also a tensor-ideal. 
    \end{enumerate}
 \end{recollection}

\begin{remark}\label{rem-identification-ImL_S}
    In practice, one often identifies the category $\T/\S$ with the subcategory $\S^\perp$ of $\T$, but also with $L_\S(\mathbb 1)\otimes \T$ by means of \cref{rec-smashing-ideals}.   
\end{remark}

We now describe a class of smashing ideals arising as \textit{finite localizations}. These smashing ideals will play a relevant role later in this manuscript. We need to set some notation first. 

\begin{notation}
    For a class  $\mathcal{E}$ of objects of $\T$ we write $\loc(\mathcal{E})$ and $\loc_\otimes(\mathcal{E})$ to denote the smallest localizing subcategory of $\T$ and the smallest localizing ideal of $\T$ containing $\mathcal{E}$, respectively.  
\end{notation}

\begin{remark}
    Note that if the class $\mathcal{E}$ is already $\otimes$-stable, then $\loc_\otimes(\mathcal{E})=\loc(\mathcal{E})$. In fact, by our assumption on $\T$, the same holds as soon as $\mathcal{E}\otimes \T^\omega\subseteq \mathcal{E}$. 
\end{remark}

\begin{recollection}\label{rec-category-thomason}
    Let $\T$ be a big tt-category, and let $Y\subseteq \Spc(\T^\omega)$ be a Thomason subset. Consider the tt-ideal 
    \[\T^\omega_Y\coloneqq \{x\in \T^\omega\mid \supp(x)\subseteq Y\} 
    \]
    of $\T^\omega$. Note that this is indeed the unique tt-ideal associated to $Y$ via Balmer's classification; see \cref{rec-Balmer-class}.
    Hence we obtain a smashing ideal $\T_Y\coloneqq\loc(\T_Y^\omega)$ of $\T$. From this, we get a localization 
    \[
    \iota_{Y}^\ast \colon\T\to \T(Y^c)
    \]
    where $\T(Y^c)$ is short for $\T/\T_Y$. In fact, this is an example of a \textit{finite localization}. Hence,  by \cite[Theorem 2.1]{Nee92} applied to the finite localization above gives us that  $\T(Y^c)$ is a rigidly-compactly generated tt-category and its rigid-compact part can be described as the idempotent completion of the Verdier quotient  $\T^\omega/\T^\omega_Y$.  Note that in this case the localization functor $\iota^\ast_{Y}$ is a geometric functor.
    Moreover,
    we obtain an exact triangle
    \[
    \e_Y \to  \mathbb 1 \to \f_Y 
    \]
    with $\f_Y\in \T(Y^c)$ and $\e_Y\in \T_Y$. As before, we can identify $\T(Y^c)$ with the category $(\T_Y)^\perp=\f_Y\otimes \T$; see  \cref{rem-identification-ImL_S}.   The objects $\e_Y$ and $\f_Y$ are  referred as left and right idempotents associated to $Y$, respectively. They satisfy that $\e_Y\otimes \e_Y\simeq \e_Y$ and $\f_Y\otimes \f_Y\simeq \f_Y$ and $\e_Y\otimes \f_Y\simeq0$.  
\end{recollection}

We conclude this subsection with the following properties satisfied by the right and left idempotents associated to different Thomason subsets. 

\begin{recollection}\label{rec-tensoring-idempotents}
   Let $Y_1,Y_2\subseteq \Spc(\T^\omega)$ be Thomason subsets. Let us recall the following properties. 
  \begin{enumerate}
      \item $\e_{Y_1}\otimes \f_{Y_2}\simeq0$ if and only if $Y_1\subseteq Y_2$. 
      \item $\e_{Y_1}\otimes \e_{Y_2}\simeq\e_{Y_1 \cap Y_2}$. 
      \item $\f_{Y_1}\otimes \f_{Y_2}\simeq\f_{Y_1\cup Y_2}$. 
  \end{enumerate}
\end{recollection}

\subsection{Purity in triangulated categories}
In this subsection we let  $\T$ denote a compactly generated triangulated category. 

\begin{recollection}
Let  $\ModT$ denote the category of (right) modules over $\T^\omega$. This is a Grothendieck category. The finitely presented objects of $\ModT$ are denoted by $\mathrm{mod}\textrm{-}\T^\omega$. 
Consider the restricted Yoneda functor 
\[
  \T\to \mathrm{Mod}\textrm{-}\T^\omega, \quad x\mapsto \Hom_\T(-,x)|_{\T^\omega}.
\]
When the context is clear we will simply write $\hat{x}$ and $\hat{f}$ to denote the image of an object $x$ and a map $f$ in $\T$ under the restricted Yoneda functor, respectively.   
\end{recollection}

\begin{warning}
    Recall that, in general, the restricted Yoneda functor is neither full nor faithful 
\end{warning}

We are now ready to recall the relevant notions about purity that we will need in the following sections.

\begin{definition}
    Let $\T$ be a compactly generated triangulated catagory. A morphism $f$ in $\T$ is called \textit{pure monomorphism} if the induced map $\hat{f}$ is a monomorphism in $\ModT$. We say that an object $x$ in $\T$ is \textit{pure-injective} if any pure monomorphism $x\to y$ in $\T$ splits. 
\end{definition}

The following lemma provides a useful characterization of pure-injective objects in a compactly generated triangulated category; see \cite[Section 1]{Kra00}. 

\begin{lemma}
     The following properties are equivalent for an object $x$ in $\T$. 
    \begin{enumerate}
        \item $x$ is pure-injective. 
        \item $\hat{x}$ is injective in $\ModT$.
        \item $\mathrm{Hom}_\T(y,x) \to \mathrm{Hom}(\hat y, \hat x)$ is a bijection for all $y$ in $\T$. 
         \item For any set $I$, we have a commutative triangle 
\begin{center}
    \begin{tikzcd}
        \coprod_I x \arrow[r, hook] \arrow[d] & \prod_I x \arrow[ld,dashrightarrow] \\ x & 
    \end{tikzcd}
\end{center}
where the vertical map is the summation map. 
    \end{enumerate}
\end{lemma}

\begin{remark}\label{pure-injectives preserved under bicon}
    From the last property we deduce that pure-injective objects are preserved under bicontinuous functors between compactly generated triangulated categories. For instance, if $f^\ast\colon \T\to \S$ is a geometric functor, then $f_\ast$ preserves pure-injective objects. 
\end{remark}

\begin{remark}
    Note that the previous lemma gives us in particular that the full subcategory of $\T$ on pure-injective objects is equivalent to the full subcategory of $\ModT$ on injective modules.  
\end{remark}

\begin{definition}
    A morphism $f\colon x \to y$ in $\T$ is called \textit{phantom} if $\hat{f}=0$ in $ \mathrm{Mod}\textrm{-}\T^\omega$.
\end{definition}

The following result describe pure-injective objects via phantom maps, more details are given in  \cite[Lemma 1.4]{Kra00}.

\begin{lemma}\label{relation between purity and phantoms}
    Let $x$ be an object in $\T$. Then $x$ is pure-injective if and only if every phantom map $f\colon y\to x$ is trivial. 
\end{lemma}

We conclude this subsection by recalling the notion of a definable subcategory of a compactly generated triangulated  category $\T$. 

\begin{recollection}\label{rec-definable-cat}
     A full subcategory $\mathcal{D}$ of $\T$ is said \textit{definable} if there is a collection of coherent functors $\{F_i\colon \T\to \mathrm{Ab}\}_{i\in I}$ such that $x$ is in $\mathcal{D}$ if and only if $F_i(x)=0$ for all $i\in I$. Recall that a functor $F\colon \T\to \mathrm{Ab}$ is \textit{coherent} if it is additive and  admits a presentation of the form 
     \[
     \mathrm{Hom}_\T(x,-)\xrightarrow[]{\Hom(f,-)} \mathrm{Hom}_\T(y,-)\to F\to 0
     \]
     where $f\colon x\to y$ is a map in  $\T^\omega$. The category of coherent functors on $\T$ is denoted by $\mathrm{Coh}(\T)$. 
\end{recollection}

\begin{remark}
    Note that a definable subcategory of $\T$ is not assumed to be triangulated. 
\end{remark}

\section{Geometric purity in tt-categories}

Let $\T$ be a rigidly-compactly generated tt-category. In this section, we study the behavior of purity with respect to a certain  class of finite localizations arising from triangular primes of $\T^\omega$. We need some preparations in order to make this precise. 

\begin{notation}\label{notation-qo}
    For a topological space $X$, we write $\mathcal{QO}(X)$ to denote the collection of quasi-compact open subsets of $X$. For a rigidly-compactly generated tt-category $\T$, we let $\mathcal{QO}(\T)\coloneqq \mathcal{QO}(\mathrm{Spc}(\T^\omega))$.
\end{notation}

\begin{notation}
 For a point  $\P$ of $\Spc(\T^\omega)$, we write $\mathrm{gen}(\P)$ to denote the generalization closure of $\P$. That is, 
     \[
     \mathrm{gen}(\P)\coloneqq \{\Q\in \Spc(\T^\omega)\mid \P\subseteq \Q\}. 
     \]
     In particular, note that $\mathrm{gen}(\P)^c=\supp(\P)$ which is Thomason subset of $\T$. 
\end{notation}

\begin{remark}\label{rem-generalizations}
     Since $\Spc(\T^\omega)$ is a spectral space, the generalization closure of a point $\P$ in $\Spc(\T^\omega)$ is given by 
    \[
    \bigcap_{\P\in U\in \mathcal{QO}(\T)} U.
    \]
\end{remark}

\begin{definition}
    Let $\P$ be point of $\Spc(\T^\omega)$. We define the \textit{tt-stalk}  of $\T$ at $\P$ to be the rigidly-compactly generated category $\T_\P\coloneqq\T(\mathrm{gen}(\P))$ associated to the Thomason subset $\supp(\P)$; see  \cref{rec-category-thomason}. 
\end{definition}

\begin{remark}
 Note that the category $\T_{\supp(\P)}$ is precisely $\loc(\P)$, and hence we can identify $\T_\P$ with the localization 
 \[
 \iota_\P^\ast\coloneqq \iota_{\supp(\P)}^\ast \colon\T\to \T/\loc(\P). 
 \]
 In particular, we can identify $\iota_{\P,\ast}\circ\iota^\ast_\P$ with $\f_{\mathrm{supp}(\P)}\otimes-$. 
\end{remark}

\begin{notation}
    We write $\f_{\P}$ to denote $\f_{\supp(\P)}$ and $\e_{\P}$ to denote $\e_{\supp(\P)}$. 
\end{notation}

Let us recall the definition of a \textit{local} rigidly-compactly generated tt-category. 

\begin{definition}
    Let $\T$ be a rigidly-compactly generated  tt-category. We say that $\T$ is \textit{local} if its rigid-compact part $\T^\omega$ is local, that is, if the zero tt-ideal $(0)$ in $\T^\omega$ is prime. 
\end{definition}

\begin{example}
    The tt-stalk $\T_\P$ is local for any point $\P$ in $\Spc(\T^\omega)$.
\end{example}

\begin{definition}\label{def-gpurity}
    Let $\T$ be a rigidly-compatly generated  tt-category. Recall the localization functor $\iota_\P^\ast\colon \T\to \T_\P$.  A map $f$ in $\T$ is called: 
\begin{enumerate}
    \item   \textit{geometrically phantom}  if  $\iota^\ast_\P(f)$ is a phantom map in $\T_\P$, for each point $\P$ of $\mathrm{Spc}(\T^\omega)$;
     \item a \textit{geometrically pure monomorphism}  if $\iota^\ast_\P(f)$ is a pure monomorphism in $\T_\P$, for each point $\P$ in $\mathrm{Spc}(\T^\omega)$;
     \item a \textit{geometrically pure epimorphism} if $\iota^\ast_\P(f)$ is a pure epimorphism in $\T_\P$, for each point $\P$ in $\mathrm{Spc}(\T^\omega)$.
\end{enumerate}
Moreover, we say that an object $x$ in $\T$ is  \emph{geometrically-pure-injective}  if any geometrically pure monomorphism $ x \to y $ in $\T$ splits.
\end{definition}

\begin{convention}
    In order to simplify terminology, we abbreviate geometrically-\textit{something} by g-\textit{something}.
\end{convention}

\begin{warning}\label{war-purity-gpurity}
While the functor $\iota_\P^\ast$ preserves pure-triangles\footnote{exact triangles where one of the maps is a pure monomorphism.}, as we will see, there is no reason to expect that it reflects pure-triangles. Later in this document, we will present an example where this indeed does not happen, thereby justifying the relevance of the previous definition.
\end{warning}

\begin{remark}\label{rem-local-gpurity}
    Note that if the category $\T$ is local, then g-purity coincides with purity. However, the converse is not true in general, as we will see in  \cref{pure=gpure affine}. 
\end{remark}

\begin{lemma}\label{relation gphan gmono gepi}
    Let $x\xrightarrow[]{f} x' \xrightarrow[]{f'} x'' \xrightarrow[]{f''} \Sigma x$ be an exact triangle in $\T$. Then the following properties are equivalent: 
    \begin{enumerate}
        \item $f$ is a g-phantom. 
        \item  $f'$ is a g-pure monomorphism.
        \item $f''$ is a g-pure epimorphism. 
    \end{enumerate}
\end{lemma}

\begin{proof}
   The same is true for phantoms, pure monomorphisms and pure epimorphisms by looking at long exact sequences induced by exact triangles  (e.g., see \cite[Lemma 1.3]{Kra00}). Hence result follows by the exactness of  the localizations $\iota_\P^\ast\colon \T\to \T_\P$.  
\end{proof}

The previous lemma gives us a relation between the g-pure-injective objects and g-phantoms:

\begin{corollary}\label{ginj kill g-phantoms}
    Let $x$ be an object in $\T$. Then $x$ is g-pure-injective if and only if every g-phantom $y\to x$ is trivial. 
\end{corollary}

\begin{proof}
    This follows from \cref{relation gphan gmono gepi} and the fact that an exact triangle splits if and only if one of the maps is trivial.  
\end{proof}

\begin{proposition}\label{prop-phantom-gphantom}
     Let $\iota^\ast\colon \T\to \U$ be a geometric functor between rigidly-compactly generated tt-categories. If $f$ is a phantom map in $\T$, then $\iota^\ast(f)$ is a phantom map in $\U$.  
\end{proposition}

\begin{proof}
Since $\iota^\ast$ is a geometric functor it preserves compact objects. Hence we obtain a functor $\ModT\to \Mod\kern-0.1em\text{-}\U^\omega$ which makes the following diagram commutative
\[
\begin{tikzcd}
\T \arrow{r}{} \arrow[swap]{d}{\iota^\ast} & \ModT \arrow{d} \\
\U \arrow{r}{} & \Mod\kern-0.1em\text{-}\U^\omega
\end{tikzcd}
\]
where the horizontal maps are the restricted Yoneda functors.  Hence, if $f$ is phantom, then $\iota^\ast(f)$ is phantom by the commutativity of the previous square.  
\end{proof}

\begin{corollary}\label{cor-phantom-gphantom}
     Let $\T$ be a big tt-category. Then any phantom map $f\colon x\to y$ in $\T$ is g-phantom.
\end{corollary}

\begin{proof}
    It is a particular case of the previous proposition applied to $\U=\T_\P$ and the localization functor $\iota_\P^\ast$. 
\end{proof}

In other words, there are more g-phantom morphisms that phantom ones. This suggests that g-purity is stronger that purity. Let us record this in the following corollary.

\begin{corollary}\label{cor-gpinj is pinj}
    If $x$ in $\T$ is a g-pure-injective, then it is pure-injective.
\end{corollary}

\begin{proof}
   Let $f\colon y\to x$ be a phantom map in $\T$. By \cref{cor-phantom-gphantom}, we obtain that the map $f$ is g-phantom. Then it follows that $f$ is trivial. The conclusion  now follows by  \cref{relation between purity and phantoms}.
\end{proof}

\begin{proposition}\label{prop-iota*-gpure-pure} 
 Let $\P\in \Spc(\T^\omega)$ and  $x$ be a pure-injective object in $\T_\P$. Then $\iota_{\P,\ast}(x)$ is g-pure-injective in $\T$. 
\end{proposition}

\begin{proof}
    Let $f\colon y\to \iota_{\P,\ast}(x)$ be a g-phantom map in $\T$. Using the adjunction $\iota_\P^\ast\dashv \iota_{\P,\ast}$ we get a morphism 
    \begin{equation}\label{eq-adjoint-f}
        \iota_\P^\ast(y)\xrightarrow[]{\iota_\P^\ast(f)} \iota_\P^\ast(\iota_{\P,\ast}(x))\xrightarrow[]{\epsilon_x} x
    \end{equation}
     where $\epsilon$ denotes the counit of the adjunction. In other words, $\epsilon\circ \iota_\P^\ast(f)$ is the unique morphism corresponding to $f$ via the adjuntion. 
     Now, since $f$ is g-phantom, $\iota_\P^\ast(f)$ is phantom. We deduce that the map from \cref{eq-adjoint-f} is phantom as well.  
     This map must be trivial since $x$ is pure-injective. Therefore $f$ must be trivial as well. We conclude  the result  by  \cref{ginj kill g-phantoms}. 
    \end{proof}

\subsection{An alternative definition of g-purity}
Before we proceed we need the following auxiliary result, which tells us that the right adjoint functor of a smashing localization of $\T$ \textit{preserves} purity. In fact, this amounts to saying that such a functor is \textit{definable}. This class of functors has been extensively studied in \cite{BW23}; their work also covers the case we are interested in. However, our proof is quite elementary, so we decided to include it. We refer the interested reader to \textit{loc. cit.} for an account of definable functors and an excellent discussion on the motivation to consider such functors.

\begin{proposition}\label{iota preserves definables}
 Let $\S$ be a smashing ideal of $\T$, and let $\iota_\S^\ast\colon \T\to \T/\S$ denote the corresponding localization functor.  Then the following properties hold. 
    \begin{enumerate}
        \item $\iota_{\S,\ast}$ preserves phantoms, pure monomorphisms, pure epimorphisms, and pure-injectives.
        \item $\iota_{\S,\ast}$ reflects pure-injective objects, that is, if $x$ is a pure-injective in the essential image of $\iota_{\S,\ast}$, then $\iota_\S^\ast(x)$ is pure-injective in $\T_\S$. 
        \item  $\iota_{\S,\ast}$ preserves and reflects definable subcategories.
    \end{enumerate}
\end{proposition}

\begin{proof}
   For part $(a)$, let $f\colon x\to y$ be a pure monomorphism in $\T/\S$. We claim that the map $\iota_{\S,\ast}(f)$ is a pure monomorphism in $\T$. Indeed,  via the naturality of the adjunction we get a commutative square for any $z\in \T^\omega$
   \begin{center}
     \begin{tikzcd}
        \Hom_\T(z,\iota_{\S,\ast}(x)) \arrow[r,"\simeq"] \arrow[d,"\widehat{\iota_{\S,\ast}(f)}(z)"'] & \Hom_{\T/\S}(\iota_\P^\ast(z),x) \arrow[d,"\hat{f}(z)"]  \\ 
        \Hom_\T(z,\iota_{\S,\ast}(y)) \arrow[r,"\simeq"] & \Hom_{\T/\S}(\iota_\P^\ast(z),y)
   \end{tikzcd}     
   \end{center}
Recall that  $\iota_\S^\ast$ preserves compact objects since it is  a geometric functor. Since  $\hat{f}$ is a monomorphism in the category of  $(\T/\S)^\omega$-modules, it follows that $\hat{f}(\iota_\P^\ast(z))$ is injective. By the commutativity of the square we deduce that $\widehat{\iota_{\S,\ast}(f)}(z)$ is injective as well. Hence  $\iota_{\S,\ast}(f)$ is a pure monomorphism in $\T$. Moreover,  the functor $\iota_{\S,\ast}$ preserves pure-injective objects by  \cref{prop-iota*-gpure-pure}. The rest follows by   \cref{relation gphan gmono gepi} and the exactness of $\iota_{\S,\ast}$.   

  Now, for part $(b)$, let $x$ be an object in $\T/\S$ such that $\iota_{\S,\ast}(x)$ is pure-injective. Consider a pure monomorphism $y\to x$ in $\T/\S$. By the first part, we have that $\iota_{\S,\ast}(y)\to \iota_{\S,\ast}(x)$ is pure monomorphism in $\T$. Since $\iota_{\S,\ast}(x)$ is pure-injective, such pure monomorphism must split. Finally, since $\iota_{\S,\ast}$ is fully faithful, such splitting occurs in $\T/\S$ as well. It follows that $x$ is pure-injective in $\T/\S$. 

  Part $(c)$ is an immediate consequence of part $(a)$ and $(b)$ since definable categories are completely determined by the pure-injective objects they contain; see \cite[Theorem 4.4]{Kra00}. 
\end{proof} 

Recall the right idempotent $\f_Y$ associated to Thomason subset $Y$ of $\Spc(\T^\omega)$; see \cref{rec-category-thomason}.

\begin{definition}[Alternative definition]\label{def-gpurity-alter}
    Let $\T$ be rigidly-compactly generated tt-category.  A map $f$ in $\T$ is called:
\begin{enumerate}
    \item   \textit{g-phantom}  if  $f\otimes \f_\P$ is a phantom map in $\T$, for each point $\P$ of $\mathrm{Spc}(\T^\omega)$;
     \item \textit{g-pure monomorphism}  if $f\otimes \f_\P$ is a pure monomorphism in $\T$, for each point $\P$ in $\mathrm{Spc}(\T^\omega)$;
     \item \textit{geometrically pure epimorphism} if $f\otimes \f_\P$ is a pure epimorphism in $\T_\P$, for each point $\P$ in $\mathrm{Spc}(\T^\omega)$.
\end{enumerate}
\end{definition}

\begin{warning}
      By \cite[Proposition 2.10]{BKS19}, we know that for any $x\in \T$, the functor $x\otimes-$ preserves pure monomorphisms, pure epimorphisms and phantom maps. However, $x\otimes-$ does not necessarily reflect purity.
\end{warning}

The following proposition shows that  \cref{def-gpurity-alter} is equivalent to  \cref{def-gpurity}. Hence, in the remainder of this manuscript, we will use them interchangeably.

\begin{proposition}
   Let $\T$ be a rigidly-compactly generated tt-category, and $\P$ be a point of $\Spc(\T^\omega)$. Let $f$ be a morphism in $\T$. Then $\iota^\ast_\P(f)$ is phantom in $\T_\P$ if and only if $f\otimes\f_\P$ is phantom in $\T$. 
\end{proposition}

\begin{proof}
    Recall  that $\f_\P\otimes-$ agrees with $\iota_{\P,\ast}\circ \iota^\ast_\P$. Then the result is a direct application of   \cref{iota preserves definables}. 
\end{proof}

\subsection{Examples of g-pure-injective objects}
We conclude this section with a couple of examples of g-pure-injective objects.
 We need some terminology from \cite{balmer2020frame} and  \cite{Bal20}. 

\begin{recollection}\label{rec:homological-res}%
 Let $\T$ be a rigidly-compactly generated tt-category.  A \textit{homological prime} for $\T^\omega$ is a maximal proper Serre $\otimes$-ideal $\mathcal{B}\subseteq \mathcal{A}^{\mathrm{fp}}$, where $\mathcal{A}\coloneqq\ModT$. The \textit{homological spectrum of $\T^\omega$} is the set 
 \[
 \mathrm{Spc}^h(\T^\omega)=\{\P\subseteq \A \mid \P \textrm{ homological prime of } \T^\omega \}
 \]
 and it is equipped with the topology generated by the basis of closed sets given by 
 \[
 \mathrm{supp}^h(x)=\{\P\in \mathrm{Spc}^h(\T^\omega)\mid \hat{x}\not \in \P \}
 \]
 for $x$ in $\T^\omega$. This construction is functorial with respect to tt-functors. That is, let $f\colon \T^\omega \to \mathcal{S}^\omega$ be a tt-functor, then we obtain a continuous map 
 \[
 \varphi(f)\colon \Spc^h(\mathcal{S}^\omega) \to \Spc^h(\T^\omega).
 \]
 Given a homological prime $\mathcal{B}$ for $\T^\omega$, we can inflate it to a localizing subcategory $\overrightarrow{\mathcal{B}}$ and consider the quotient $\A/\overrightarrow{\mathcal{B}}$. This quotient is symmetric monoidal, and taking the injective envelope of the monoidal unit in this quotient will determine a pure-injective object $E_\mathcal{B}$ in $\T$, such that $\overrightarrow{\mathcal{B}}=\mathrm{Ker}(\hat{E}_\mathcal{B}\otimes -)$.
  In general, there is a continuous map 
\[
\phi\colon \mathrm{Spc}^h(\T^\omega)\to \mathrm{Spc}(\T^\omega), \quad \P\mapsto \h^{-1}(\P).
\]
This map is always surjective.
\end{recollection}

\begin{proposition}
    Let $\T$ be a big tt-category, and $\mathcal{B}$ be homological prime of $\T^\omega$. Then the pure-injective object $E_{\mathcal{B}}$ is g-pure-injective. 
\end{proposition}

\begin{proof}
    First, consider a homological prime $\mathcal{B}$ lying above a triangular prime $\P$. The localization functor $\iota_\P^\ast\colon \T\to \T_\P$ induces a map 
    \[
  \varphi\colon   \Spc^h(\T_\P^\omega)\to \Spc^h(\T^\omega).  
    \]
    Let $\mathcal{B'}\in \varphi^{-1}\mathcal{B}$. Then by \cite[Lemma 5.6]{Bal20b} we obtain that $E_\mathcal{B}$ is a direct summand of $\iota_{\P,\ast}(E_{\mathcal{B}'})$. Hence the result follows by  \cref{prop-iota*-gpure-pure} and the fact that $\mathrm{Im}(\iota_{\P,\ast})=\loc(\P)^\perp$ is a thick subcategory of $\T$. 
\end{proof}

In \cref{g-pure vs g-pure á la PS}, we will present an example where purity and geometric purity do not coincide. Before that, let us conclude this section with an instance in which they do coincide.

\begin{proposition}\label{pure=gpure affine}
Let $R$ be a commutative ring. Then the pure-injective objects and the geometrically pure-injective objects in the derived category $\mathbf{D}(R)$ coincide.
\end{proposition}

\begin{proof}
 Let $f\colon x\to y$ be a g-pure monomorphism in $\mathbf{D}(R)$. We claim that $f$ is pure monomorphism. Note that this claim gives us that the two notions of purity in $\mathbf{D}(R)$ coincide.  In order to prove the claim, we need to verify that 
 \[
 \mathrm{Hom}_{\mathbf{D}(R)}(z,x)\xrightarrow[]{\hat{f}_z} \mathrm{Hom}_{\mathbf{D}(R)}(z,y)
 \]
 is injective for all $z\in \mathbf{D}(R)^\omega$. Recall that $\hat{f}_z$ denotes $\mathrm{Hom}(z,f)$. Let $\pp$ be a prime of $R$. Then the  localization at $\pp$ gives us a map $$\mathrm{Hom}_{\mathbf{D}(R)}(z,x)\otimes_R R_\pp\xrightarrow[]{\hat{f}_z\otimes_R R_\pp} \mathrm{Hom}_{\mathbf{D}(R)}(z,y)\otimes_R R_\pp$$ but since $z$ is compact in $\mathbf{D}(R)$, we  have isomorphisms: 
 \begin{enumerate}
     \item $\mathrm{Hom}_{\mathbf{D}(R)}(z,x)\otimes_R R_\pp\cong\mathrm{Hom}_{\mathbf{D}(R_\pp)}(z_\pp,x_\pp)$, and
     \item $\mathrm{Hom}_{\mathbf{D}(R)}(z,y)\otimes_R R_\pp\cong\mathrm{Hom}_{\mathbf{D}(R_\pp)}(z_\pp,y_\pp)$.
 \end{enumerate}
 Now, we can use that  $f$ is g-pure monomorphism to obtain that $f\otimes_R R_\pp$ is pure monomorphism in $\mathbf{D}(R_\pp)$. It follows that $\hat{f}_z\otimes_R R_\pp$ is injective, for each prime $\pp$ of $R$. Therefore $\hat{f}_z$ is injective as well. The object $z \in  \mathbf{D}(R)^\omega$ was arbitrary and so $\hat{f}$ is a monomorphism.
\end{proof}

\begin{remark}
    Note that the above example shows that, even in non-local rigidly-compactly generated tt-categories, the notions of purity and g-purity may coincide.
\end{remark}

\section{Purity at closed points and open covers}

One might wonder whether it is possible to test geometric purity on a smaller class of localizations. In this section, we address this question for the class of smashing ideals arising from closed points of the tt-spectrum. We also relate the theory of geometric purity in $\T$ with that in the categories $\T(U)$, where $U$ ranges over the quasi-compact open subsets of the Balmer spectrum of $\T$.

But first, let us begin with the following auxiliary result that tells us that the localizations $\iota_\P^\ast$ detect whether a g-pure-injective object is trivial. In principle, a non-trivial object in $\T$ might be trivial under $\iota_\P^\ast$, for all $\P \in \Spc(\T^\omega)$.

\begin{lemma}\label{g-pure embedding for any object}
Let $x$ be an object  in $\T$. Then there is a g-pure-injective object  $y$ and a g-pure monomorphism $x\to y$. In fact, one can take $y$ as a product of indecomposable g-pure-injective objects. 
\end{lemma}

\begin{proof}
    For each prime $\P$ in $\T^\omega$, consider a pure monomorphism $\iota_\P^\ast(x)\xrightarrow[]{\phi_\P} E(\P)$ with $E(\P)$ a product of indecomposable pure-injective objects in the local category $\T_\P$. We claim that the natural map 
   \begin{equation}\label{eq-map1}
        x\to \prod_{\P \in \mathrm{Spc}(\T^\omega)} \iota_{\P,\ast}(E(\P))
   \end{equation}
    satisfies the desired properties. Recall that $\iota_{\P,\ast}$ preserves products and indecomposable objects. Therefore, \cref{prop-iota*-gpure-pure} gives us that each $\iota_{\P,\ast}(E(\P))$ is a product of indecomposable g-pure-injective objects. Hence, we only need to verify that the above map is a g-pure monomorphism. Now, for a prime $\mathcal{Q}$ in $\T^\omega$, consider the following commutative diagram 
    \[
\begin{tikzcd}
\iota_{\mathcal{Q}}^\ast(x) \arrow{r}{} \arrow[swap]{rdd}{\phi_\mathcal{Q}} \arrow[swap]{dr}{} & \iota_{\mathcal{Q}}^\ast(\iota_{\mathcal{Q},\ast}E(\mathcal{Q}))\oplus \iota_{\mathcal{Q}}^\ast(z) \arrow{d}{} \\
 & \iota_{\mathcal{Q}}^\ast(\iota_{\mathcal{Q},\ast}E(\mathcal{Q})) \arrow{d}{\simeq}\\
 & E(\mathcal{Q})
\end{tikzcd}
\]
where the horizontal map is obtained by applying $\iota^\ast_\Q$ to the map in \cref{eq-map1}, the top vertical map is the projection map and 
\[
z= \prod_{\P \in \mathrm{Spc}(\T^\omega), \P\not=\mathcal{Q}} \iota_{\P,\ast}(E(\P))).
\]
Since $\phi_\mathcal{Q}$ is a pure monomorphism, we deduce that the horizontal  morphism is a pure monomorphism as well. This completes the result. 
\end{proof}

\begin{corollary}\label{trivial g-inj are detected by localizitions}
    Let $x$ be a g-pure-injective object in $\T$. Then $x\simeq 0$ if and only if $\iota_\P^\ast(x)\simeq 0$ for all $\P\in\Spc(\T^\omega)$. 
\end{corollary}

\begin{proof}
    If $\iota_\P^\ast(x)\simeq 0$, then we can take $E(\P)=0$ in the proof of  \cref{g-pure embedding for any object}. In particular, if $\iota_\P^\ast(x)\simeq 0$ for all $\P\in\Spc(\T^\omega)$, then the zero morphism $x\to 0$ is a g-pure monomorphism. But this morphism has to split since  $x$ is g-pure-injective. Thus $x$ must be trivial. 
\end{proof}

\subsection{Closed points} For a Thomason subset $Y\subseteq\Spc(\T^\omega)$, recall the associated smashing localization $\iota_{Y}^\ast \colon\T\to \T(Y^c)$. See  \cref{rec-category-thomason} if needed.

\begin{remark}\label{rem:restriction functors}
    Let $U$ and $V$ the complement of Thomason subsets of $\Spc(\T^\omega)$. If $U\subseteq V$ then we have a commutative diagram 
    \begin{center}
        \begin{tikzcd}
            &  \T \arrow[rd,"\iota_{U^c}^\ast"] \arrow[ld,"\iota_{V^c}^\ast"'] & \\
            \T(V) \arrow[rr,"\mathrm{res}^{V,\ast}_U"] &  & \T(U)
        \end{tikzcd}
    \end{center}
    where $\mathrm{res}^{V,\ast}_U$ is the functor $-\otimes \f_{U^c}$. See \cite{BF11} or \cref{subsection-smashingloc}. 
\end{remark}

\begin{remark}
    For a quasi-compact open subset $U$ of $\mathrm{Spc}(\T^\omega)$, note that $\T(U)$ is a particular case of  \cref{rec-category-thomason} applied to the Thomason subset $U^c$. 
\end{remark}

\begin{lemma}\label{lem: loc at U then at p}
     Let $U$ be a quasi-compact open subset of $\mathrm{Spc}(\T^\omega)$ and $\P$ be a prime in $\T^\omega$ lying in $U$. Then the localization functor $\iota_\P^\ast\colon \T\to \T_\P$ factors through $\T(U)$. That is, we get a commutative diagram  
     \begin{center}
         \begin{tikzcd}
        \T  \arrow[r,"{\iota_U^\ast}"] \ar[dr,"\iota_\P^\ast"'] &  \T(U) \ar[d] \\
         & \T_\P
     \end{tikzcd} 
     \end{center}
     where the vertical map is tensoring with $\f_{\mathrm{gen}(\P)^c}$.
\end{lemma}

\begin{proof}
    This is a particular case of  \cref{rem:restriction functors}, and the alternative description of $\T_\P$ by noticing that $\mathrm{gen}(\P)$ must be contained in $U$; see  \cref{rem-generalizations}.
\end{proof}

We are in a position to show that geometric purity in  $\T$ only depends on localizations at closed points of the Balmer spectrum of $\T$. More concretely: 
 
\begin{proposition}\label{prop-purity-closed-point}
 Let $\T$  be a big tt-category. Then a morphims $f\colon x\to y$ in $\T$ is $g$-phantom if and only if it $\iota_\mathcal{M}^\ast(f)$ is phantom in $\T_\mathcal{M}$, for each closed point $\mathcal{M}$ of $\Spc(\T^\omega)$.
\end{proposition}

\begin{proof}
 Assume that $f$ is such that $\iota_\mathcal{M}^\ast(f)$ is phantom in $\T_\mathcal{M}$ for all closed points $\mathcal{M}$ of $\Spc(\T^\omega)$. Fix $\P \in \Spc(\T^\omega)$. We need to verify that $\iota_\P^\ast(f)$ is phantom. For this, recall that $\P$ must contain a closed point, say $\mathcal{M}$. Then we have the relation $ \mathrm{gen}(\P)\subseteq \mathrm{gen}(\mathcal{M})$. By  \cref{rem:restriction functors}, we obtain a commutative diagram 
 \begin{center}
        \begin{tikzcd}
            &  \T \arrow[rd,"\iota_{\mathrm{gen}(\P)^c}^\ast"] \arrow[ld,"\iota_{\mathrm{gen}(\mathcal{M})^c}^\ast"'] & \\
            \T(\mathrm{gen}(\mathcal{M})) \arrow[rr,"\mathrm{res}^{\mathrm{gen}(\mathcal{M}),\ast}_{\mathrm{gen}(\P)}"] &  & \T(\mathrm{gen}(\P)).
        \end{tikzcd}
    \end{center} 
    In other words, the functor $\iota^\ast_{\mathrm{gen}(\P)}=\iota_\P^\ast$ factors through $\T(\mathrm{gen}(\mathcal{M}))=\T_\mathcal{M}$. By  \cref{prop-phantom-gphantom}, we get that  $\mathrm{res}^{\mathrm{gen}(\mathcal{M}),\ast}_{\mathrm{gen}(\P)}$ preserves phantoms. We conclude that $\iota_\mathcal{P}^\ast(f)$ is phantom as we wanted.  
\end{proof}

\subsection{Covers by quasi-compact opens} We now turn to the relationship between geometric purity in $\T$ and that in $\T(U)$, where $U$ is a quasi-compact open subset of the Balmer spectrum of $\T^\omega$.

\begin{proposition}\label{prop-iota-gp-gp}
     Let $U$ be a quasi-compact open subset of $\Spc(\T^\omega)$. Let $f$ be a g-phantom in $\T$. Then $\iota_{U^c}^\ast(f)$ is g-phantom in $\T(U)$. 
\end{proposition}

\begin{proof}
    This follows from  \cref{lem: loc at U then at p} and the fact that $\Spc(\T(U)^\omega) \cong U$ via the induced map.
\end{proof}

\begin{proposition}\label{prop-iotaU-gp-gp}
     Let $U$ be a quasi-compact open subset of $\Spc(\T^\omega)$. Then the functor $\iota_{U^c,\ast}$ preserves g-pure-injective objects. 
\end{proposition}

\begin{proof}
Let $x$ be a g-pure-injective object in $\T(U)$.  Consider a g-phantom map $f\colon y\to \iota_{U^c,\ast}(x)$ in $\T$. We will show that $f$ is trivial and hence the result will follow by  \cref{ginj kill g-phantoms}. By  adjunction, we obtain a unique map 
\[
\Bar{f}\colon\iota_{U^c}^\ast(y)\to \iota_{U^c}^\ast(\iota_{U^c,\ast}(x))\cong x
\]
where the isomorphism is the counit since  $\iota_{U^c,\ast}$ is fully faithful. Now, we claim that $\Bar{f}$ is $g$-phantom in $\T(U)$. But this is  \cref{prop-iota-gp-gp}.  
\end{proof}

We conclude this section with the following lemma, which will play a crucial role in the rest of this work.

\begin{lemma}\label{lem-purity-covers}
    Let $\{U_i\}_{i\in I}$ be a finite cover of $\mathrm{Spc}(\T^\omega)$ by quasi-compacts opens.  Let $x$ be an object in $\T$. Write $\iota_i^\ast$ and $\f_{i}$ to denote $\iota_{U_i^c}^\ast$ and $\f_{U_i^c}$, respectively.  Then the following properties hold.
    \begin{enumerate}
        \item The natural map 
         \[
         \psi\colon x\to \bigoplus_{i\in I}  \f_i\otimes x
         \] 
         is a g-pure monomorphism.

        \item If  $x$ is an  indecomposable g-pure-injective object in $\T$, then  there is an isomorphism $x\simeq \f_i \otimes x$ for some $i\in I$.
    \end{enumerate}
\end{lemma}

\begin{proof}
    For the first claim, let $\P$ be a point in $\mathrm{Spc}(\T^\omega)$ and assume that it lies in $U_k$ for $k\in I$. By  \cref{lem: loc at U then at p}, we get
    \begin{equation}\label{eq:iotak}
        \iota_{\P}^\ast (x)\simeq \iota_{\P}^\ast(\f_k\otimes x).
    \end{equation}
    Thus,  we obtain that the map
    \[
    \iota_\P^\ast(\psi)\colon \iota_\P^\ast(x)\to \iota_{\P}^\ast(\f_k\otimes x)\oplus \bigoplus_{i\in I\backslash\{k\}} \iota_\P^\ast (\f_i\otimes  x)
    \]
    is pure monomorphism. Indeed, composition with the projection to the component corresponding to $k$ is pure monomorphism using  \cref{eq:iotak}. We deduce that $\psi$ is a g-pure monomorphism in $\T$ as we wanted. 

    Now, assume that $x$ is an indecomposable g-pure-injective object in $\T$. Then   the natural map $\psi$ has to split since $x$ is g-pure-injective. 
    Since the Krull-Schmidt property holds for pure-injective objects (see Theorems E.1.23 and E.1.24 in \cite{Prest})
    we deduce that $x$ has to be a direct summand of  $\f_i\otimes x$, for some $i\in I$. This implies that $x\in \T(U_i)$. Recall that we can identify $\T(U_i)$ with $\f_i\otimes \T$; see \cref{rec-category-thomason}. Therefore $x\simeq \f_i\otimes x$, as we wanted. 
\end{proof}

\section{Reduction to stalks}

In this section, we show that any indecomposable g-pure-injective object of $\T$ arises as the pushforward of a pure-injective object of the tt-stalk $\T_\P$ for some triangular prime $\P$. Our proof relies on some higher categorical results; hence, from now on, we will assume that $\T$ is a rigidly-compactly generated stable $\infty$-category. We will keep the terminology and simply refer to $\T$ as a rigidly-compactly generated tt-category. 

The main result of this section is the following.

\begin{theorem}\label{thm-gp-stalks} 
   Let $\T$ be a rigidly-compactly generated tt-category, and let $x$ be an indecomposable g-pure-injective object in $\T$. Then there is a prime $\P$ in $\T^\omega$ and a pure-injective object $y$ in $\T_\P$ such that $\iota_{\P,\ast}(y)\simeq x$.  
\end{theorem}

We need some preparations. In particular, the following is a straightforward observation but it will play a relevant role in our proof of  \cref{thm-gp-stalks}. 

\begin{lemma}\label{lem-intersections}
    Let $X$ be a spectral space. Let $I$ be a set, and  consider an arbitrary  collection $\{U_i\}_{i\in I}$ of quasi-compact open  subsets of $X$ with the finite intersection property; that is, any finite subfamily of $\{U_i\}_{i\in I}$ has non-empty intersection. Then $\bigcap_{i\in I}U_i\not= \varnothing$. 
\end{lemma}

\begin{proof}
  Recall that the Hochster dual $X^\vee$ of $X$ is a spectral space with the same underlying set as $X$, and whose open subsets are the Thomason subsets of $X$. In particular, note that our  collection of quasi-compact open subsets $\{U_i\}_{i \in I}$ becomes a  collection of closed subsets in $X^\vee$ with the finite intersection property. The claim then follows from the well-known reformulation of compactness in terms of families of closed subsets satisfying the finite intersection property.
\end{proof}

Let us record the following relevant description of the idempotent $\f_\P$, which is a reformulation of \cite[Lemma 6.6]{Ste13} and \cite[Lemma 2.3.6]{Ste11}. This reformulation also appears in \cite[Corollary 3.10]{balchin2024profinite}. Recall that $\mathcal{QO}(\T)$ denotes the set of quasi-compact open subsets of $\mathrm{Spc}(\T^\omega)$; see \cref{notation-qo}.

\begin{lemma}\label{stalks as a colimit}
    Let $\T$ be a rigidly-compactly generated tt-category and $\P\in \Spc(\T^\omega)$. Then the natural map  
    \[
     \underset{\P\in U\in \mathcal{QO}(\T)}{\colim}  \f_{U^c}\to \f_\P
    \]
    is an equivalence. 
\end{lemma}

With the necessary preliminaries in place, we are now in a position to prove the announced theorem.

\begin{proof}[Proof of  \cref{thm-gp-stalks}]
    Let $\{U_i\}_{i\in I}$ be a finite cover of the triangular spectrum $\mathrm{Spc}(\T^\omega)$ by quasi-compacts open subsets. By  \cref{lem-purity-covers}, we get that there is $i\in I $ such that $x\simeq  \f_{U^c_i}\otimes x$. Now, recall that $\Spc(\T(U_i))\cong U_i$. Hence, replacing $x$ by $\iota_{U^c_i}^\ast x$, and $\T$ by $\T(U_i)$ we can iterate the same argument; namely, we can cover $U_i$ by quasi-compact open subsets $\{U_{i,j} \}_{j\in I_i}$ of $U_i$, and apply  \cref{lem-purity-covers} to obtain $j\in I_i$ such that 
    \[
\f_{U^c_{i,j}}\otimes \iota_{U^c_i}^\ast x\simeq  \iota_{U^c_i}^\ast x.
    \]
    Repeating   the same process, we obtain a nested collection $\{U_\alpha\}_{\alpha\in \mathcal{I}}$ of non-empty quasi-compact opens of $\Spc(\T^\omega)$. By  \cref{lem-intersections}, there exists $\P\in \bigcap_{\alpha\in \mathcal{I}}U_\alpha$. We claim that $x\simeq \f_\P\otimes x$.  Indeed, 
    note that 
    \[
    \f_\P\simeq\underset{\P\in U\in \mathcal{QO}(\T)}{\colim}  \f_{U^c}\simeq \underset{\alpha\in \mathcal{I}}{\colim}  \f_{U_\alpha^c}
    \]
   where the first equivalence is \cref{stalks as a colimit}, and the second follows from the fact that the collection $\{U_\alpha\}_{\alpha \in \mathcal{I}}$ is final in the diagram obtained from $\{ U \in \mathcal{QO}(\T) \mid \P \in U \}$. It follows that
    \[
    \f_{\mathrm{gen}(\P)}\otimes x\simeq \colim_{\alpha\in \mathcal{I}}(\f_{U_\alpha^c})\otimes x\simeq \colim_{\alpha\in \mathcal{I}}(\f_{U_\alpha^c}\otimes x).
    \]
     It remains to show that the canonical map 
    \[
    \colim_{\alpha\in \mathcal{I}}(\f_{U_\alpha^c}\otimes x)\to x
    \]
    is an equivalence. This follows from the construction of the $U_\alpha$, which tells us that the diagram $\alpha\mapsto  (\f_{U_\alpha^c})\otimes x$ is constant with value $x$. Therefore, the colimit must be equivalent to $x$.
\end{proof}

We now present a series of straightforward consequences of the previous result. But first let us recall some terminology and set some notation.

\begin{recollection}
Recall that $\pinj(\T)$ denotes the set of isomorphism classes of indecomposable pure-injective objects in $\T$. The \textit{Ziegler spectrum} of $\T$ is the topological space whose underlying set is $\pinj(\T)$ and whose closed subsets are those of the form $\mathcal{D} \cap \pinj(\T)$, where $\mathcal{D}$ is a definable subcategory of $\T$. This space is denoted by $\mathrm{Zg}(\T)$. We refer to \cite{Kra00} for further details. 
\end{recollection}

\begin{definition}
Let $\gpinj(\T)$ denote the set of isomorphism classes of indecomposable g-pure-injective objects in $\T$. The \textit{geometric subspace}  $\mathrm{GZg}(\T)$ of $\mathrm{Zg}(\T)$ is by definition $\gpinj(\T)$ with the subspace topology. 
\end{definition}

\begin{lemma}\label{prop-closed-emb}
 Let $\S$ be a smashing ideal of $\T$. Then   the functor $\iota_{\S,\ast}\colon \T/\S\to \T$ induces  a  closed embedding
         \[
         \mathrm{Zg}(\T/\S)\to \mathrm{Zg}(\T)
         \]
    between Ziegler spectra. 
\end{lemma}

\begin{proof}
    This is immediate from   \cref{iota preserves definables}. 
\end{proof}

\begin{corollary}\label{coro-surjection-gzg-stalks}
    Let $\T$ be a rigidly-compactly generated tt-category. Then the following properties hold. 
    \begin{enumerate}
        \item The class of g-pure-injective objects and pure-objects if coincide if and only if the induced map 
    \[
   \coprod\mathrm{Zg}(\iota_{\P,\ast}) \colon \coprod_{\P\in \mathrm{Spc}(\T^\omega)}\mathrm{Zg}(\T_\P)\to \mathrm{Zg}(\T)
    \]
    is surjective. 
    \item The map 
    \[
   \coprod\mathrm{Zg}(\iota_{\P,\ast}) \colon \coprod_{\P\in \mathrm{Spc}(\T^\omega)}\mathrm{Zg}(\T_\P)\to \mathrm{GZg}(\T)
    \]
    is a quotient.
    \item Let $\{U_i\}_{i\in I}$ be a finite cover of $\Spc(\T^\omega)$ by quasi-compact open subsets. Then the induced map 
\[
\coprod \mathrm{GZg}(\iota_{U^c_i,\ast})\colon \coprod_{i\in I} \mathrm{GZg}(U_i) \to \mathrm{GZg}(\T)
\]
is a quotient.  
    \end{enumerate}
     Here the coproducts are taken in the category of topological spaces. 
\end{corollary}

\begin{proof}
    Note that the image of $\mathrm{Zg}(\iota_{\P,\ast})$ actually lands in the  geometric subspace  $\mathrm{GZg}(\T)$ by  \cref{prop-iota*-gpure-pure}. Hence part $(a)$ follows by  \cref{thm-gp-stalks}. 

     For part $(b)$, again  by  \cref{thm-gp-stalks} we know that the map is surjective. The rest follows since each of the components is a closed map by  \cref{prop-closed-emb}. 

     For part $(c)$, note that the maps $\mathrm{GZg}(\iota_{U^c_i,\ast})$ are well defined by  \cref{prop-iotaU-gp-gp} and are closed by  \cref{prop-closed-emb}. Hence it remains to show that $\coprod \mathrm{GZg}(\iota_{U_i,\ast})$ is surjective. We have the following diagram. \begin{center}
        \begin{tikzcd}
         \coprod_{i\in I} \coprod_{\P\in U_i} \mathrm{Zg}(\T_\P)    \arrow[r] \arrow[rr, bend left=20, "\coprod\coprod \mathrm{Zg}(\iota_{\P,\ast})"]  & \coprod_{i\in I} \mathrm{GZg}(U_i) \arrow[r] & \mathrm{GZg}(\T).
        \end{tikzcd}
    \end{center}  
    Note that the top arrow is surjective since it contains at least the map from part $(a)$.  Hence the result follows. 
\end{proof}

We obtain the following result as an immediate consequence. 

\begin{corollary}\label{coro-gzg-closed}
    Let $\{U_i\}_{i\in I}$ be a finite cover of $\Spc(\T^\omega)$ by quasi-compact open subsets. Assume that each $\mathrm{GZg}(\T(U_i))$ is a closed subset of the Ziegler spectrum $\mathrm{Zg}(\T(U_i))$. Then $\mathrm{GZg}(\T)$ is a closed subset of $\mathrm{Zg}(\T)$. 
\end{corollary}

\begin{proof}
   By  \cref{coro-surjection-gzg-stalks}$(c)$ we deduce that 
   \[
   \mathrm{GZg}(\T) =\bigcup_{i\in I} \mathrm{Im}(\mathrm{GZg}(\iota_{U_i,\ast})).
   \] 
   Since each map $\mathrm{Zg}(\iota_{U_i,\ast})\colon \mathrm{Zg}(\T(U_i))\to \mathrm{Zg}(\T)$ is  closed, we obtain the result.  
   \end{proof}

In particular, the previous result applies for qcqs schemes: 

\begin{corollary}\label{coro-zgz-scheme}
    Let $X$ be a quasi-compact quasi-separated scheme. Then the set $\mathrm{GZg}(\mathbf{D}_{\textrm{qc}}(X))$ is  a closed subset of $\mathrm{Zg}(\mathbf{D}_{\textrm{qc}}(X))$.
\end{corollary}

\begin{proof}
    Consider an affine open cover $\{U_i\}_{i\in I}$ of $X$. By  \cref{pure=gpure affine}, we have that 
    \[
    \gpinj(\mathbf{D}_{\textrm{qc}}(U_i))=\pinj(\mathbf{D}_{\textrm{qc}}(U_i)).
    \]
    Hence the conclusion follows from  \cref{coro-gzg-closed}.
\end{proof}

\section{Geometric purity in the derived category of the projective line}\label{g-pure vs g-pure á la PS}
In this section, we show that geometric purity can differ from purity in rigidly-compactly generated tt-categories. To this end, we compare geometric purity in the derived category $\mathbf{D}_{\textrm{qc}}(X)$ of a quasi-compact, quasi-separated scheme $X$ with the notion of geometric purity in the category of quasi-coherent sheaves $\mathrm{QCoh}(X)$ as defined by Prest and Slavík. The main case covered in this section is when $X = \mathbb{P}^1 \coloneqq \mathbb{P}^1_k$, where $k$ denotes an arbitrary field.

For convenience, we recall the relevant notions and terminology from \cite{SP18}.

\begin{definition}
    Let $X$ be a quasi-compact quasi-separated scheme with structure sheaf $\mathcal{O}_X$. A short exact sequence of $\mathcal{O}_X$--modules 
    \[
    0 \to F\to G\to H\to 0
    \]
    is \textit{geometrically-pure} (or just \textit{g-pure}) if it remains exact after tensoring with any $\mathcal{O}_X$--module $L$. Equivalently, if the exact sequence 
    \[
    0\to F_x\to G_x\to H_x\to 0
    \]
    is a pure exact sequence of $\mathcal{O}_{X,x}$--modules, for all $x\in X$. With this notion in place, we can now define \textit{g-pure monomorphisms}, \textit{g-epimorphisms}, and \textit{g-phantom maps}, as well as \textit{g-pure-injectives}, in a manner analogous to the corresponding notions in a locally finitely presented category.  
\end{definition}

 Let $\mathrm{QCoh}(X)$ denote the category of quasi-coherent sheaves over $X$. Recall that $\mathrm{QCoh}(X)$ is a full subcategory of $\mathcal{O}_X\textrm{-}\mathrm{Mod}$. The notion of geometric purity in the category $\mathrm{QCoh}(X)$ is defined through the one in $\mathcal{O}_X\textrm{-}\mathrm{Mod}$:

\begin{definition}
 A quasi-coherent sheaf $F$ is \textit{g-pure-injective} if $\mathrm{Hom}_{\mathrm{QCoh}(X)}(-,F)$ preserves g-pure exact sequences of quasi-coherent sheaves.    
\end{definition}

From now on, we focus on the projective line $\mathbb{P}^1$ over an arbitrary field, in order to show that geometric purity differs from purity in $\mathbf{D}(\mathbb{P}^1)\coloneqq\mathbf{D}(\mathrm{QCoh}(\mathbb{P}^1))$.

\begin{example}
  Let   $a,b,c,d\in \mathbb{Z}$ such that $a < b < d$, $a < c < d$ and
$a + d = b + c$.    By  \cite[Example 5.1]{SP18}, we know that there is a short exact sequence of quasi-coherent sheaves 
\begin{equation}\label{non g-pure exact sequence}
    0\to \mathcal{O}(a)\to \mathcal{O}(b)\oplus\mathcal{O}(c)\to \mathcal{O}(d)\to 0
\end{equation} which is not split. Moreover, this sequence cannot be pure exact since $\mathcal{O}(d)$ is finitely presented. Nevertheless, this sequence is g-pure exact since passing to  any stalk gives us a short exact sequence of free modules.

Now, consider the short exact sequence \cref{non g-pure exact sequence} as an exact triangle in $\mathbf{D}(\mathbb{P}^1)$. We claim that this triangle is not pure exact. Suppose that it is a pure exact triangle. 
By \cite[Lemma 2.4]{GP05} we know that 
\[
0\to H^n(\mathcal{O}(a))\to H^n( \mathcal{O}(b)\oplus\mathcal{O}(c))\to H^n(\mathcal{O}(d))\to 0
\]
is a pure exact sequence, for all $n\in \mathbb Z$. But we already observed that this sequence for $n=0$ is not pure exact. This completes  the claim.  On the other hand, this exact triangle under localization at a point $x\in \mathbb{P}^1$, gives us the exact triangle induced by the short exact sequence 
\[
0\to \mathcal{O}(a)_x\to \mathcal{O}(b)_x\oplus\mathcal{O}(c)_x\to \mathcal{O}(d)_x\to 0
\]
of free $\mathcal{O}_{X,x}$--modules. Hence we deduce that sequence \cref{non g-pure exact sequence} induces a g-pure exact triangle in $\mathbf{D}(\mathbb{P}^1)$ that is not pure exact. 
\end{example}

We can go further: all indecomposable g-pure-injective objects in $\mathbf{D}(\mathbb{P}^1)$ can be explicitly described. To begin, let us recall the description of pure-injective objects in $\mathrm{QCoh}(\mathbb{P}^1)$.

\begin{recollection}
 For a closed point $x$ in $\mathbb{P}^1$, let 
 \[
 \iota_{x,\ast}\colon \mathrm{Mod}(\mathcal{O}_{\mathbb{P}^1,x})\to \mathrm{QCoh}(\mathbb{P}^1)
 \]
 denote the push forward along the inclusion map of schemes $\mathrm{Spec}(\mathcal{O}_{\mathbb{P}^1,x})\to \mathbb{P}^1$. Let $\mathfrak{m}_x$ denote the maximal ideal of $\mathcal{O}_{\mathbb{P}^1,x}$ and $k(x)$ the residue field $\mathcal{O}_{\mathbb{P}^1,x}/\mathfrak{m}_x$. Then define the sheaves  
\[
\mathcal{E}(x)\coloneqq \iota_{x,\ast}(E(x)),\quad \mbox{and} \quad \mathcal{A}(x)\coloneqq\iota_{x,\ast}(A(x))
\]
where $E(x)$ denotes the injective envelope of the residue field $k(x)$ and $A(x)$ the $\mathfrak{m}_x$--adic completion of $\mathcal{O}_{\mathbb{P}^1,x}$. Finally, recall that the sheaf of rational functions $k(\eta)$ is the residue field of  the generic point $\eta$ of $\mathbb{P}^1$.  
\end{recollection}

\begin{recollection}
    The indecomposable  pure-injective sheaves in $\mathrm{QCoh}(\mathbb{P}^1)$ are described in the following list given in \cite[Proposition 4.4.1]{KS17}: 
\begin{enumerate}
    \item Indecomposable coherent sheaves. 
    \item Pr\"ufer sheaves $\mathcal{E}(x)$, for $x\in \mathbb{P}^1$ closed.
    \item Adic sheaves $\mathcal{A}(x)$, for $x\in \mathbb{P}^1$ closed. 
    \item The sheaf of rational functions $k(\eta)$.
\end{enumerate}
\end{recollection}

\begin{recollection}\label{pure-injectives in D(P)}
    Let $A$ be a hereditary ring. By \cite[Theorem 8.1]{GP05},  the set of isomorphism classes of indecomposable pure injective objects of $\mathbf{D}(A)$ is in bijection with a countable disjoint union of copies of the set of isomorphism classes of indecomposable pure-injective right $A$-modules 
    \[
    \pinj(\mathbf{D}(A)) =\coprod_{n\in \mathbb{Z}}\pinj(\mathrm{Mod}\text{-}A),
    \]
    where the correspondence is given by $x_n[n]\mapsfrom x_n$, for $x_n$ in the $n$th copy of $\pinj(\mathrm{Mod}\text{-}A)$.
\end{recollection}

\begin{remark}
    Recall that there is a derived equivalence 
    \[
    \mathbf{D}(R)\simeq \mathbf{D}(\mathbb{P}^1)
    \]
    where $R$ is the path algebra of the Kronecker quiver $(\cdot \rightrightarrows \cdot )$.   Since $R$ is hereditary, we obtain a description of all the pure-injective objects in $\mathbf{D}(\mathbb{P}^1)$ as above. 
\end{remark}

\begin{remark}\label{rem-o(n)-no-gpure}
    Note that the sheaves $\mathcal{O}(n)$ are pure-injective objects in $\mathrm{QCoh}(\mathbb{P}^1)$ for any $n\in \mathbb{Z}$, but not g-pure-injective by means of the exact sequence from \cref{non g-pure exact sequence}. 
\end{remark}

\begin{proposition}\label{prop-o(n)-notgpure}
    The complex $\mathcal{O}(n)[n]$ is pure-injective in $\mathbf{D}(\mathbb{P}^1)$, but not g-pure-injective. 
\end{proposition}

\begin{proof}
    The g-pure exact triangle given by \cref{non g-pure exact sequence} is not split and hence $\mathcal{O}(n)$ cannot be g-pure-injective.
\end{proof}

\begin{recollection}\label{geometric pure injectives in P}
    By \cite[Theorem 5.5]{SP18}, we know that any indecomposable pure-injective sheaf in $\mathrm{QCoh}(\mathbb{P}^1)$ belongs to the following disjoint list.
    \begin{enumerate}
        \item Line bundles, i.e., $\mathcal{O}(n)$ for $n\in \mathbb{Z}$. 
        \item  g-pure-injective sheaves: These correspond to the indecomposable torsion sheaves, the Pr\"ufer sheaves, the adic sheaves and the sheaf of rational functions. 
    \end{enumerate}
\end{recollection}

\begin{notation}
Let us  write $\pinj(\mathbb
P^1)$ and $\gpinj(\mathbb
P^1)$ to denote $\pinj(\mathrm{QCoh}(\mathbb{P}^1))$ and $\gpinj(\mathrm{QCoh}(\mathbb{P}^1))$, respectively.    
\end{notation}

\begin{proposition}
    Let $z$ be a complex in $\mathbf{D}(\mathbb{P}^1)$. Then $z$ is g-pure-injective if and only if it is of the form $\coprod_{n\in \mathbb{Z}} x_n[n] $ with $x_n$ a g-pure-injective quasi-coherent sheaf on $\mathbb{P}^1$. In particular we have a set bijection 
\[
\gpinj(\mathbf{D}(\mathbb
P^1))=\coprod_{\mathbb{Z}}\gpinj(\mathbb{P}^1), \quad x_n[n] \mapsfrom x_n.
\]
\end{proposition}

\begin{proof}
    First, by \cite[Theorem 8.1]{GP05} we already know that any pure-injective object $z$ in $\mathbf{D}(\mathbb
P^1)$ has the form $\coprod_{n \in \mathbb{Z}} x_n[n]$, where each $x_n$ is a pure-injective quasi-coherent sheaf on $\mathbb{P}^1$. It is straightforward to see that any shift of a g-pure-injective complex is again g-pure-injective. Moreover, as observed in \cite[Lemma 3.9]{GP04}, a coproduct of the form $\coprod_{n\in \mathbb{Z}} x_n[n]$ is isomorphic to the product $\prod_{n\in \mathbb{Z}} x_n[n]$. It follows that if each $x_n[n]$ is g-pure-injective, then so is $\coprod_{n\in \mathbb{Z}} x_n[n]$. Therefore, we may assume that $z$ is indecomposable, in which case it must be of the form $x[n]$ for some indecomposable pure-injective quasi-coherent sheaf $x$.

Note that $x$ cannot be a line bundle, by  \cref{prop-o(n)-notgpure}. Thus, it remains to verify that $z$ is g-pure-injective whenever $x$ is a g-pure-injective sheaf. But by \cite{SP18}, we know that such sheaves are precisely the pushforwards of pure-injective $\mathcal{O}_{\mathbb{P}^1,x}$-modules for some $x \in \mathbb{P}^1$. The result then follows from  \cref{prop-iota*-gpure-pure}.
\end{proof}

\section{Local to global for spatiality of the  frame of smashing ideals} 

In this section, we leverage  \cref{thm-gp-stalks} to show that, under certain conditions, the spatiality of the frame of smashing localizations of a rigidly-compactly generated tt-category $\T$ is a local property on $\Spc(\T^\omega)$. Let us fix some terminology first.  

\begin{notation}
    Let $\T$ be a rigidly-compactly generated  tt-category. We write $\mathbb{S}^\otimes(\T)$ to denote the frame of smashing ideals of $\T$.  
\end{notation}

\begin{terminology}
    We say that a rigidly-compactly generated  tt-category $\T$ is \emph{spatial} if its frame of smashing ideals $\mathbb{S}^\otimes(\T)$ is spatial.
\end{terminology}

The first observation is that spatiality passes to local categories without much difficulty. 

\begin{proposition}\label{prop-global-to-local}
    Let $\T$ be a rigidly-compactly generated tt-category. If $\T$ is spatial, then the tt-stalk  $\T_\P$ is  spatial for each point $\P$ of $\mathrm{Spc}(\T^\omega)$.
\end{proposition}

\begin{proof}
   Fix $\P\in \mathrm{Spc}(\T^\omega)$. Observe that the interval 
   \[[\P,\T]\coloneqq \{\S\in\mathbb{S}^\otimes (\T)\mid\P\subseteq \S\}\]
    is a frame. Indeed, this follows since it is a quotient given by the nucleus 
    \[ \_\vee\P\colon\mathbb{S}^\otimes(\T)\rightarrow[\P,\T].
    \]
    Moreover, we obtain an equivalence
    \[[\P,\T]\simeq\mathbb{S}^\otimes (\T_{\P}), \quad  \S \mapsto \S/\loc(\P) \]
 given by \emph{taking the quotient}. Furthermore, the nucleus $- \vee \P$ is complemented in the frame of all nuclei on $\mathbb{S}^\otimes(\T)$. Consequently, its associated sublocale is complemented, since it is closed. Now, suppose that $\T$ is spatial. Then, by \cite[Proposition 3.3]{picado2011frames}, which states that complemented sublocales of a spatial locale are spatial, we conclude that $\T_\P$ is spatial.
\end{proof}

We will prove a \textit{partial} converse to the previous proposition. Some preliminary results are needed.

\begin{recollection}
     Recall that a definable subcategory  $\mathcal{D}$ of a compactly generated triangulated category $\T$ is a full subcategory $\D$ such that  there is a collection of coherent functors $\{F_i\colon \T\to \mathrm{Ab}\}_{i\in I}$ such that $x$ is in $\mathcal{D}$ if and only if $F_i(x)=0$ for all $i\in I$; see \cref{rec-definable-cat}. A \textit{definable tt-ideal} $\mathcal{D}$ in a rigidly-compactly generated tt-category $\T$ is a definable subcategory $\mathcal{D}$ which is also a thick tensor-ideal, that is, it is a triangulated subcategory of $\T$ which satisfies that $\T\otimes \D\subseteq \D$. 
\end{recollection}

\begin{remark}
    Since a definable subcategory $\D$ of $\T$ is closed under homotopy filtered colimits, we get that the property $\T\otimes\D\subseteq \D$ is equivalent to $\T^\omega\otimes \D\subseteq \D$. 
\end{remark}

\begin{notation}
    Let $\T$ be a rigidly-compactly generated tt-category. Let $(\mathbb{D}^\otimes(\T),\subseteq)$  denote the lattice of definable tensor-ideals in $\T$ ordered by inclusion. 
\end{notation}
 
The following result corresponds to \cite[Proposition 5.2.13]{Wag23}.

\begin{proposition}\label{equiv between def and smashing}
    There is a morphism of lattices 
\begin{align*}
  \mathbb{S}^\otimes(\T) & \to  \mathbb{D}^\otimes(\T)^{\mathrm{op}}\\
        \mathcal{S} & \mapsto  \mathcal{S}^\perp   
\end{align*}
    with inverse given by $\mathcal{D}\mapsto {}^\perp\mathcal{D}$. Here ${}^\perp\mathcal{C}$ (resp. $\mathcal{C}^\perp$) denotes the left (resp. right) orthogonal complement of $\mathcal{C}$.      
\end{proposition}

\subsection{The geometric Ziegler spectrum}

Recall that $\gpinj(\T)$ denotes the collection of isomorphism classes of indecomposable g-pure-injective objects in $\T$. Note that this is a set since every g-pure-injective is pure-injective by  \cref{cor-gpinj is pinj}, and the collection of isomorphism classes of pure-injectives has been shown to form a set; see \cite{Kra00}.   

\begin{definition}
    We say that a subset $X$ of $\gpinj(\T)$ is \textit{tt-Ziegler closed} (or just \textit{tt-closed})  if $X=\mathcal{D}\cap \gpinj(\T)$ for some definable tt-ideal $\mathcal{D}$ of $\T$. 
\end{definition}

\begin{remark}
    We emphasize that we do not mean that these  tt-closed subsets of $\gpinj(\T)$ necessarily define the closed subsets of  a topology on $\gpinj(\T)$.
\end{remark}

\begin{convention}\label{conv-geometric-Ziegler}
    Whenever the tt-closed subsets define a topology on $\gpinj(\T)$, we denote the resulting space by 
    \[
    \mathrm{ttZg}(\T)
    \]
    and call it the \emph{geometric tt-Ziegler spectrum of} $\T$.  
\end{convention}

\begin{remark}
    Note that when $\T$ is local and the tt-closed subsets define a topology on $\gpinj(\T)$, the space $\mathrm{ttZg}(\T)$ has the same underlying set of points as $\mathrm{Zg}(\T)$. However, the two topologies are usually different.
\end{remark}

\begin{remark}
     Assume that the  tt-closed subsets define a topology on $\gpinj(\T)$. In this case, the map  \begin{align*}
       \mathbb{D}^\otimes(\T) & \to \mathrm{Closed}(\mathrm{ttZg}(\T))\\ 
        \D & \mapsto \D\cap \gpinj(\T)
   \end{align*} 
  is an order-preserving bijection. Indeed, note that this map is the restriction of the map involved in the fundamental correspondence in \cite{Kra00}, which states that any definable subcategory is completely determined by the pure-injectives it contains. In particular, this map is a lattice isomorphism.
\end{remark}

\begin{terminology}\label{def-ziegler-realize}
   In the case that tt-closed subsets define a topology on $\gpinj(\T)$, we simply say that the geometric tt-Ziegler spectrum $\mathrm{ttZg}(\T)$ \textit{realizes} the lattice of definable ideals $\mathbb{D}^\otimes(\T)$.
\end{terminology}

Let us emphasize that the difficulty in showing that the tt-closed subsets determine a topology on the geometric Ziegler spectrum lies in proving that the union of two tt-closed subsets is again tt-closed. Indeed, tt-closed subsets are always closed under intersections, as we will see.

\begin{lemma}\label{lemma-intersection tt-closed}
    Let $\T$ be a rigidly-compactly generated tt-category. Let $(X_i)_{i\in I}$ be a family of tt-closed subsets of $\gpinj(\T)$. Then $\cap_{i\in I}X_i$ is tt-closed. 
\end{lemma}

\begin{proof}
    First,  note that the intersection of elements in $\mathbb{D}^\otimes(\T)$ is an element in $\mathbb{D}^\otimes(\T)$ since the intersection of definable subcategories (resp. tt-ideals) is again definable (resp tt-ideal). Then, since   $X_i=\D_i \cap \gpinj(\T)$, for some $\D_i\in \mathbb{D}^\otimes(\T)$, we obtain that 
 \[
 \bigcap_{i\in I} X_i = \gpinj(\T)\cap \bigcap_{i\in I}\D_i.
 \]
We conclude that the set of tt-closed subsets is closed under intersection. 
\end{proof}

\begin{proposition}\label{local to global}
    Let $\T$ be a rigidly-compactly generated tt-category.  Suppose that the geometric tt-Ziegler spectrum of $\T_\P$ realizes the lattice of definable tensor-ideals $\mathbb{D}^\otimes(\T_\P)$ for each $\P\in\Spc(\T^\omega)$ (see \cref{def-ziegler-realize}). Then the tt-closed subsets of $\gpinj(\T)$ define a topology, that is, the geometric tt-Zielger spectrum $\mathrm{ttZg}(\T)$ realizes $\mathbb{D}^\otimes(\T)$.  
\end{proposition}

\begin{proof}
 We need to verify that the class of tt-closed subsets of $\gpinj(\T)$ is closed under arbitrary intersections and finite unions. The former is the content of \cref{lemma-intersection tt-closed}.
 
On the other hand, let $X_i=\D_i\cap \gpinj(\T)$ with $\D_i\in \mathbb{D}^\otimes(\T)$, for $i=1,2$. Consider 
\[
\D=\D_1\vee \D_2,
\]
where $\vee$ denotes the join in $\mathbb{D}^\otimes (\T)$. In order to prove our claim, it is enough to verify that  $X_1\cup X_2=\D\cap \gpinj(\T)$. It is clear that $X_1\cup X_2\subseteq \D\cap \gpinj(\T)$. Let $x \in \D\cap \gpinj(\T)$. By \cref{equiv between def and smashing}, we have that 
\[
\mathrm{Hom}_\T(\S_1\cap \S_2,x)=0, \quad \mbox{where}\quad \S_i={}^\bot \D_i,\, i=1,2.
\]
Moreover, by  \cref{thm-gp-stalks}, there is a prime $\P$ in  $\T^\omega$, and a pure-injective $y$ in $\T_\P$ such that $\iota_{\P}^\ast(y)\simeq x $. In particular, we have that 
\[
0=\mathrm{Hom}_\T(\S_1\cap \S_2,x)\cong\mathrm{Hom}_{\T_\P}(\iota_\P^\ast \S_1\cap \iota^\ast_{\P}\S_2,y).
\]
By hypothesis, $\mathbb{D}^\otimes(\T_\P)$ determines the closed subsets of $\mathrm{ttZg}(\T_\P)$ (see \cref{conv-geometric-Ziegler}), hence we deduce that  
\[
y \in \left( (\iota^\ast_{\P}\S_1)^\bot \vee (\iota^\ast_{\P}\S_1)^\bot \right) \cap \mathrm{Zg}^\otimes(\T_\P).
\]
Therefore  
\[
0=\mathrm{Hom}_{\T_\P}(\iota_\P^\ast \S_i, y)\cong \mathrm{Hom}_\T(\S_i,x) \mbox{ for }  i=1  \mbox{ or } i=2.
\]
It follows that $y$ is in $X_1\cup X_2$ as we wanted. 
\end{proof}

The previous proposition provides a type of local-to-global reduction concerning the spatiality of a rigidly-compactly generated tt-category. More concretely:

\begin{corollary}
Let $\T$ be a rigidly-compactly generated tt-category. If the geometric tt-Ziegler spectrum of $\T_\P$ realizes $\mathbb{D}^\otimes(\T_\P)$ for each $\P\in\Spc(\T^\omega)$, then $\T$ is spatial. In this case, $\mathbb S^\otimes(\T)$ is isomorphic to the frame of open subsets of $\mathrm{ttZg}(\T)$.   
\end{corollary}

\begin{proof}
This follows by  \cref{equiv between def and smashing} together with  \cref{local to global}. 
\end{proof}

In a similar fashion as in  \cref{local to global} we have a local to global statement in terms of open covers.

\begin{proposition}\label{local to global covers}
 Let $\T$ be a rigidly-compactly generated tt-category, and let $\{U_i\mid i\in I\}$ be a finite cover  of $\mathrm{Spc}(\T^\omega)$ by quasi-compact open subsets. Suppose that the geometric tt-Ziegler $\mathrm{ttZg}(\T(U_i))$ realizes  $\mathbb{D}^\otimes(\T(U))$ for each $i\in I$. Then $\mathbb{D}^\otimes(\T)$ determines the closed subsets of a topology on $\gpinj(\T)$. 
\end{proposition}

\begin{proof}
    This follows a similar pattern as the proof of  \cref{local to global} using  \cref{coro-surjection-gzg-stalks}$(c)$ instead of  \cref{thm-gp-stalks}. 
\end{proof}

\begin{corollary}
    With the same hypothesis as in  \cref{local to global covers}, we obtain that $\T$ is spatial.  In particular, $\mathbb S^\otimes(\T)$ is isomorphic to the frame of open subsets of $\mathrm{ttZg}(\T)$.   
\end{corollary}

\section{Examples: spatial tt-categories via the geometric Ziegler spectrum}\label{sec-examples}

In this section, we provide several examples in which the geometric tt-Ziegler spectrum of a rigidly-compactly generated tt-category realizes the frame of smashing ideals.
 
\begin{recollection}\label{rec-vonneumann}
Suppose that $A$ is a Hereditary or a von Neumann regular ring. Garkusha and Prest \cite[Theorem 8.1 and Theorem 8.5]{GP05} gave a full description of the pure injective indecomposable complexes in $\mathbf{D}(A)$, namely
\[\mathrm{pinj}(\mathbf{D}(A)) = \coprod_{n\in \Z}\pinj(\mathrm{Mod}\text{-}A)_n\,,\]
where $\pinj(\mathrm{Mod}\text{-}A)_n$ denotes the set of isomorphism classes of indecomposable pure-injective $A$-modules, and the subscript keeps track of the $n$th copy of this set.
\end{recollection}

\begin{remark}
    In the context of \cref{rec-vonneumann}, if we suppose further that $A$ is commutative, by  \cref{pure=gpure affine} we have \[\gpinj(\mathbf{D}(A)) = \pinj(\mathbf{D}(A)).\]  
\end{remark}

Let us begin with the semisimple case:

\begin{example}\label{Zieg geom semisimple} Let $R$ be a semisimple commutative ring. Then the derived category of $R$ is equivalent to $(\mathrm{Mod}\text{-}R)^\Z$. It follows that the indecomposable pure injectives in $\mathbf{D}(R)$ are completely determined by the indecomposable pure injective right $R$-modules. 

We claim that under these hypotheses the geometric tt-closed subsets define a topology on $\gpinj(\mathbf{D}(R))$.  Since $R$ generates $\mathbf{D}(R)$ as localizing subcategory, we get that any thick subcategory of $\mathbf{D}(R)$ is automatically stable under the tensor product with compact objects. Moreover, since $R$ is semisimple, cones in $\mathbf{D}(R)$ are direct sums and hence a thick subcategory is simply a subcategory closed under suspensions and direct sums and retracts. We conclude that a definable $\otimes$-ideal is simply a definable subcategory closed under suspension. 

Now,  Wagstaffe in \cite[Proposition 6.1.5]{Wag23} showed that the sets $\D\cap\pinj(\T)$ define a topology on $\pinj(\T)$, where $\T$ is a rigidly-compactly generated tt-category and $\D$ is a definable suspension closed subcategory; see also \cref{rec-shiftziegler}. This proves our claim. 
\end{example}

We now address von Neumann regular rings and Dedekind domains, respectively.

\begin{example}
 Let $R$ be a commutative von Neumann regular ring. Then the localization rings $R_\pp$ are fields for every prime ideal $\pp\subset R$. Under the homeomorphism $\mathrm{Spec}(R)\cong \Spc(\mathbf D(R)^\omega)$ we may identify $\pp$ with its image $\P$ in the Balmer spectrum $\Spc(\mathbf D(R)^c)$, and then we have an equivalence $\mathbf D(R)_\P\cong \mathbf D(R_\pp)$. By \cref{Zieg geom semisimple} and  \cref{local to global} we have that $\mathrm{ttZg}(\mathbf{D}(R))$ realizes the lattices of definable tensor ideals of $\mathbf{D}(R)$, and as a consequence, the frame of smashing ideals of $\mathbf{D}(R)$ is isomorphic to the frame of opens of $\mathrm{ttZg}(\mathbf{D}(R))$. Note that the telescope conjecture is known to hold in this case \cite{bazzoni2017smashing}.
\end{example}

\begin{example}\label{examp-Ziegler Dedekind} Let $R$ be a commutative Dedekind domain. In this case, $\mathbf{D}(R)$ satisfies the telescope conjecture. Hence by \cite[Theorem 3.15]{Thomason97} there is a bijective correspondence between the definable thick subcategories of $\mathbf{D}(R)$ and the Thomason's subsets of $\mathrm{Spec}(R)$. Specializing this bijection to the local rings $R_\pp$, for every $\pp\in \mathrm{Spec}(R)$ we obtain that there are only 2 non-trivial definable thick subcategories in $\mathbf{D}(R_\pp)$, as $\mathrm{Spec}(R_\pp)$ is a Sierpiński space. But we know that at least we have $\mathbf{D}(R_\pp)$ and the full subcategory $\mathbf{D}(Q(R))$, where $Q(R)$ is the ring of fractions of $R$. Hence this is the full list. Then
\begin{align*}
    \mathbf{D}(R_\pp)\cap \pinj(\mathbf{D}(R_\pp)) = \pinj(\mathbf{D}(R_\pp)), \quad \mbox{ and } \\    \mathbf{D}(Q(R))\cap \pinj(\mathbf{D}(R_\pp)) = \coprod_{n\in \Z} Q(R)[-n]
\end{align*}
is clearly a topology, yielding $\mathrm{ttZg}(\mathbf{D}(R_\pp))$. Moreover, these closed sets (plus the empty set) are in a lattice bijection with the closed sets in the Sierpiński topology, hence it realizes the frame of definable thick subcategories. By  \cref{local to global} we obtain $\mathrm{ttZg}(\mathbf{D}(R))$ realizes the frame of smashing ideals of $\mathbf{D}(R)$.
\end{example}

\begin{proposition}
Let $k$ be a field, and $\mathbb P^1$ be the projective line over $k$. Then the  geometric  closed subsets of $\gpinj(\mathbb P^1)$ define a topology. In particular the frame of definable $\otimes$-ideals is realized by the geometric tt-Ziegler spectrum $\mathrm{ttZg}(\mathbf{D}(\mathbb
P^1))$.
\end{proposition}

\begin{proof}
First note that $\mathrm{ttZg}(\mathbf{D}(k[x]))$ is well defined by \cref{examp-Ziegler Dedekind}. Then  \cref{local to global covers} applied to the standard affine cover of $\mathbb P^1$ gives the desired result. 
\end{proof} 

In order to generalize the previous examples, let us recall the following results.

\begin{recollection}\label{rec-coherent}
    Let $R$ be a commutative coherent ring. Let $\mathrm{inj}_\mathrm{zig}(R)$ denotes the set of iso-classes of indecomposable injective $R$--modules equipped with the Ziegler topology.  
    Then  the map 
    \[\mathrm{Spec}(R)^{\vee} \to \mathrm{inj}_\mathrm{zig}(R), \, \pp\mapsto E(R/\pp) \]
  is injective, where $E(-)$ denotes injective hull and $\mathrm{Spec}(R)^{\vee}$ denotes the Hochster dual of the Zariski spectrum of $R$; see \cite[Section 3.6]{PW24} for a more detailed discussion. Furthermore, by \cite[14.4.5]{Prest}, this map induces an isomorphisms of frames of open subsets. 
\end{recollection}

\begin{recollection}\label{rec-shiftziegler}
    Let $\T$ be a rigidly-compactly generated tt-category. A definable category $\D$ is \textit{shift-closed} if $\Sigma^n \D\subseteq \D$ for any integer $n$. In \cite[Proposition 6.1.5]{Wag23}, it is shown that the shift-closed definable subcategories of $\T$ induce a topology on $\pinj(\T)$; the resulting topological space is denoted by $\mathrm{Zg}^\Sigma(\T)$. Now, let $\mathrm{Zg}(\T)/\Sigma$ be the quotient of $\mathrm{Zg}^\Sigma(\T)$ by the obvious action of the shift functor $\Sigma$. Then the quotient map induces an isomorphism of frame of opens 
    \[
    \mathrm{Op}(\mathrm{Zg}(\T)/\Sigma)\simeq \mathrm{Op}(\mathrm{Zg}^\Sigma(\T)).
    \]
\end{recollection}

\begin{recollection}\label{rec-emb-rings}  
   Let $R$ be a ring.  Write $\mathrm{Zg}(R)$ to denote the Ziegler spectrum of $R$; see for instance \cite[Section 6]{GP05}. Then the functor $\Mod(R)\to \mathbf{D}(R)$ which maps a module to a complex concentrated in degree 0, induces a closed embedding 
   \[
   \mathrm{Zg}(R) \to \mathrm{Zg}(\mathbf{D}(R))/\Sigma.
   \]
   Indeed, this follows from \cite[Theorem 7.3]{GP05} combined with \cref{rec-shiftziegler}.
\end{recollection}

\begin{proposition}
Let $R$ be a commutative coherent ring that satisfies the telescope conjecture. Then $\mathrm{ttZg}(\mathbf{D}(R))$  realizes the frame of smashing subcategories.
\end{proposition}

\begin{proof}
Combining \cref{rec-coherent} and \cref{rec-emb-rings} we obtain a closed embedding 
\[\mathrm{Spec}(R)^{\vee} \to \mathrm{inj}_{\mathrm{zig}}(R) \to \mathrm{Zg}(R) \to \Zg(\mathbf{D}(R))/\Sigma. \] 
By our assumption on the telescope conjecture for $\mathbf{D}(R)$, we get that $\mathrm{Cl}(\mathrm{Spec}(R)^{\vee})$ is isomorphic to $\mathbb{D}^\otimes(\mathbf{D}(R))$; see \cref{rem-hochster-dual} and \cref{equiv between def and smashing}. Hence the embedding above, shows that direct image induces a map  of lattices 
\begin{equation}\label{eq-incl}
   \mathbb{D}^\otimes(\mathbf{D}(R)) \to \mathbb{D}^\Sigma(\mathbf{D}(R)) 
\end{equation}
where $\mathrm{\mathbb{D}^\Sigma(\mathbf{D}(R))}$ denotes the lattice of shift-closed definable subcategories of $\mathbf{D}(R)$ which is isomorphic to $\mathrm{Cl}(\Zg(\mathbf{D}(R))/\Sigma)$; see \cref{rec-shiftziegler}. 
We claim that this implies the desired result. Note that it is enough to verify that finite union of tt-Ziegler closed subsets is tt-Ziegler closed. Let $X_i=\D_i\cap \gpinj(\mathbf{D}(R))$ with $\D_i\in \mathbb{D}^\otimes(\mathbf{D}(R))$, for $i=1,2$. Consider 
\[
\D=\D_1\vee \D_2,
\]
where $\vee$ denotes the join in $\mathbb{D}^\otimes (\mathbf{D}(R))$. In order to prove our claim, it is enough to verify that  $X_1\cup X_2=\D\cap \gpinj(\mathbf{D}(R))$. But using the map from \cref{eq-incl}, we deduce that the join $\D=\D_1\vee \D_2$ is the same as the join in $\mathbb{D}^\Sigma (\mathbf{D}(R))$. Hence, by \cite[Proposition 6.1.5]{Wag23} we deduce that $X_1\cup X_2=\D\cap \gpinj(\mathbf{D}(R))$ as we wanted. 
\end{proof}


\bibliographystyle{amsalpha}
\bibliography{mybibfile}

\end{document}